\documentclass[letterpaper,11pt]{article}

\usepackage{endnotes}
\usepackage{comment}
\usepackage{amssymb}
\usepackage{amsfonts}
\usepackage{amsmath}
\usepackage{amsthm}
\usepackage{graphicx,color,xcolor}
\definecolor{db}{rgb}{0.2, 0.46, 0.63}

\usepackage{comment}
\usepackage{array}

\usepackage{url}

\usepackage{setspace}

\usepackage{natbib}
\setcitestyle{numbers}
\setcitestyle{square}

\def\R{\mathbb{R}}

\def\sF{{\mathcal F}}

\def\sA{{\mathcal A}}

\def\E{\mathbb{E}}

\def\sF{\mathcal{F}}
\def\P{\mathbb{P}}

\def\sE{{\mathcal E}}
\def\N{\mathbb N}

\numberwithin{equation}{section}
\theoremstyle{plain}                
\newtheorem{theorem}{Theorem}[section]

\newtheorem{corollary}[theorem]{Corollary}
\theoremstyle{definition}           

\newtheorem{example}[theorem]{Example}

\newtheorem{assumption}[theorem]{Assumption}

\theoremstyle{remark}               

\usepackage[margin=1.25in]{geometry}

\newcommand{\Pp}{\mathbb{P}}
\newcommand{\dd}{\,\mathrm{d}}

\usepackage{amsmath,amssymb,amsthm,mathtools}
\usepackage{enumitem}
\usepackage{microtype}
\usepackage{color}
\usepackage[hidelinks]{hyperref}

\title{Strict SDE Comparison for Cusp Coefficients and Counterexamples}
\author{}
\date{\today }

\begin{document}
\begin{center}
\large{\bf Strict SDE Comparison for Cusp Coefficients and Counterexamples}\footnote{ Kasper Larsen is corresponding author and has contact information: Email: \url{KL756@math.rutgers.edu} and mailing address: Department of Mathematics, Rutgers University, Hill Center 330 - Busch Campus, 110 Frelinghuysen Road, Piscataway, NJ 08854-8019, USA. 
}

\ \\

{\large \bf Kasper Larsen}\\
Rutgers University

\ \\ 

\today

\ \\

\end{center}

\begin{abstract}
We provide a tractable sufficient condition for strict comparison for solutions of the one-dimensional stochastic
differential equation
\[
    \dd X_t^x=\sigma(X_t^x)\,\dd B_t, \quad X_0^x=x\in I=(\ell,r),
\]
for $\sigma>0$ and continuous. Our proof is based on a two-dimensional Lyapunov argument, which allows us to prove strict comparison for some coefficients in $W^{1,p}_{\text{loc}}(I)$, $1\le p<2$. We illustrate using $\sigma(x):=1+|x|^\beta$ for $x\in\R$, $\beta \in (0,1)$, and show that strict comparison holds if and
only if $\beta\in[\frac12,1)$.  We give examples showing that neither
\(\sigma\in W^{1,p}_{\mathrm{loc}}(\mathbb R)\), \(p\in(1,2)\),
nor
\(\sigma\in C^\beta(\mathbb R)\), \(\beta\in[\frac12,1)\),
is sufficient for strict comparison, even when combined with
boundedness, uniform ellipticity, global strong existence, and
pathwise uniqueness.
\end{abstract}

\section{Introduction} Let \(I=(\ell,r)\) be an open interval and consider the one-dimensional SDE
\begin{equation}\label{eq:intro-sde}
    \dd X_t^x=\sigma(X_t^x)\,\dd B_t,
    \qquad X_0^x=x\in I.
\end{equation}
Here $\sigma:I\to(0,\infty)$ is continuous, and we assume strong
existence and pathwise uniqueness up to the lifetime $\xi^x$ in $I$.
In particular, $\sigma$ is bounded away from zero on each compact
subinterval of $I$. For $x<y$, comparison asserts $X_t^x\le X_t^y$
up to the minimum of their lifetimes. Strict comparison asks that
\[
    \P\!\left(X_t^x<X_t^y\text{ for all }0\le t<\xi^x\wedge\xi^y\right)=1.
\]
Strict comparison is also called strong comparison, nonconfluence, and noncoalescence. 

Pathwise uniqueness ensures that, should two solutions \(X^x\) and \(X^y\) with \(x<y\) meet, they cannot later separate. Under the conditions described below, local-time methods yield pathwise uniqueness. Then, continuity and pathwise uniqueness imply comparison but not strict comparison in general. Yamada and Watanabe (1971) ensure pathwise uniqueness whenever, locally,
\begin{align}\label{LeGall1}
    |\sigma(x)-\sigma(y)|
    \le \rho(|x-y|),
\quad     \int_{0+}\frac{\dd u}{\rho(u)^2}=\infty,
\end{align}
where the modulus function $\rho:[0,\infty)\to[0,\infty)$ is increasing with $\rho(0)=0$ and $\rho(u)>0$ for all $u>0$. Nakao (1972, 1983) and Le Gall (1983) give pathwise-uniqueness criteria. The following sufficient condition is from Le Gall (1983).  For each compact
interval $J\subset I$, it suffices there exists  an increasing function
$f_J:J\to\R$ such that ($f_J$ can be discontinuous)
\begin{equation}\label{NakaoLeGall}
    |\sigma(x)-\sigma(y)|^2
    \le
    |f_J(x)-f_J(y)|,
    \qquad x,y\in J,\quad J\subset I \text{ compact.}
\end{equation}
In particular, all coefficients $\sigma \in C_\text{loc}^\beta(I)$ with  
$\beta\ge\frac12$ as well as all $\sigma \in W^{1,2}_{\mathrm{loc}}(I)
$ satisfy \eqref{NakaoLeGall}  and so pathwise uniqueness and comparison hold.
However, for $\sigma \in C_\text{loc}^\beta(I)$ with $\beta \in (0,\frac12)$, the condition in  \eqref{NakaoLeGall}  can fail.\footnote{This failure is not if and only if. In Section \ref{sec:counterexample}, we consider $\sigma_{\text{sym}}(x):=1+(|x|\land 2)^\beta$ for $x\in\R$ and $\beta \in (0,\frac12)$. This function belongs to $C_\text{loc}^\beta(\R)$ and it satisfies \eqref{NakaoLeGall}.}

Next, we discuss the literature related to strict comparison and homeomorphic stochastic flow properties. On one hand, an example like $\sigma(x) := \sqrt{|x|}$ for $x\in \R$ satisfies \eqref{LeGall1} but fails strict comparison because, eventually, all paths get absorbed at zero. On the other hand, several sufficient conditions for strict comparison are known. Theorem~1.1 in Yamada and Ogura (1981) gives strict comparison under the modulus condition
\begin{align}\label{YOcond1}
    \int_{0+}\frac{u\,\dd u}{\rho(u)^2}=\infty.
\end{align}
The local-Lipschitz result in Remark~2.1 in Yamada and Ogura (1981) implies strict comparison for $\sigma\in\operatorname{Lip}_{\mathrm{loc}}(I)$. This includes $\sigma(x):=\sqrt{x}$ on $I=(0,\infty)$.  While \eqref{YOcond1} does not cover all
$W^{1,2}_{\mathrm{loc}}(I)$ coefficients --- for example,
$\sigma(x):=1+|x|^\beta$ with $\beta\in(\frac12,1)$ ---
Theorem~2.1 of Yamada (1986) relaxes the \eqref{YOcond1} through a factorization 
\begin{align}\label{Yamada1986}
    \sigma(x)=\sigma_1(t,x)\sigma_2(x),
\end{align}
among whose hypotheses are that $\sigma_1$ satisfies  \eqref{YOcond1} and
$\sigma_2$ is absolutely continuous with $\sigma_2'\in L^2$.
In particular, after localization, 
Yamada (1986) covers all positive coefficients in
$W^{1,2}_{\mathrm{loc}}(I)$.\footnote{In Appendix \ref{outsideW1p}, the coefficient $\sigma_{\mathrm{lac}}$ defined in \eqref{eq:example-lacunary} satisfies \eqref{YOcond1}, and hence is
covered by Yamada (1986), but $\sigma_{\mathrm{lac}}$ does not belong to
$W^{1,2}_{\mathrm{loc}}(\R)$. }  For the SDE \eqref{eq:intro-sde}, the sufficient condition in  Ouknine and Rutkowski (1990)  is  equivalent to $\sigma\in W^{1,2}_\text{loc}$ and is therefore covered by Yamada (1986).\footnote{This equivalence with $\sigma \in W^{1,2}_\text{loc}$ is a consequence of Remark~3.4 in Ouknine and Rutkowski (1990). }  The coefficient  $\sigma_\text{sym}(x):=1+ \sqrt{|x|}$ for $x\in \R$ is in $C^{\frac12}_\text{loc}(\R)$ but $\sigma_\text{sym} \notin W^{1,2}_\text{loc}(\R)$.  While $\sigma_\text{sym}$ satisfies both \eqref{LeGall1} and \eqref{NakaoLeGall}, it fails both \eqref{YOcond1} and the conditions in Yamada (1986). Because Yamada (1986) covers strictly more than coefficients in  $W^{1,2}_{\mathrm{loc}}(I)$, we prove that $\sigma_\text{sym}$ is not covered by Yamada (1986) in Appendix \ref{Yamada_fails} (this observation is consistent with Remark 3.2 in Yamada, 1986).  Nevertheless, Example \ref{ex:sqrt-3} shows that strict comparison holds for  $\sigma_\text{sym}$.

Theorem~C in Fang and Zhang (2005) gives a multidimensional
non-confluence theorem, however, in our setting \eqref{eq:intro-sde},
their modulus conditions do not relax \eqref{YOcond1}. A similar observation applies to Theorem~1.9 in Lan and Wu (2014).   After specializing to \eqref{eq:intro-sde}, Ren and Zhang (2024) assume boundedness, uniform continuity, and uniform nondegeneracy, and impose two alternative assumptions on $\sigma$ in their Theorem~5.1. Their first option is \eqref{YOcond1} together with an additional mild restriction on
$u/\rho(u)$ as $u\downarrow 0$.   Their second option is the Ouknine–Rutkowski condition, which in our setting \eqref{eq:intro-sde} is equivalent to $\sigma\in W^{1,2}_\text{loc}$.  Thus, for the driftless
one-dimensional  \eqref{eq:intro-sde}, none of these results enlarges the class already covered by Yamada (1986). 

Zhang (2005), with the correction in Zhang (2006), studies homeomorphic
stochastic flows under a log-Lipschitz-type modulus. In our setting \eqref{eq:intro-sde}, the diffusion assumption in Theorem~4.1 of Zhang (2005) implies \eqref{YOcond1}. The homeomorphism conclusions of He and Zhang (2007) use the same stronger log-type modulus, which ---  in our setting \eqref{eq:intro-sde} --- implies \eqref{YOcond1}.  For comparison, Theorem 7 in Flandoli, Gubinelli and Priola (2010) establishes a
stochastic flow of diffeomorphisms when $\sigma\in C_b^3$ is uniformly nondegenerate. Their conclusion gives a $C^1$ stochastic flow of diffeomorphisms and strict comparison follows.

Zhang (2011) proves strong well-posedness and a stochastic
homeomorphism-flow property for SDEs with singular drift and
Sobolev diffusion coefficients. In dimension one, Theorem~1.1
of Zhang (2011) allows $1<p<2$. However, in the proof, the
Krylov estimate is applied to squared diffusion-gradient
quantities and, for $1<p<2$, the spatial exponent $\frac p2$
lies outside the range for which the Krylov estimate holds.  This concern does not apply to Zhang (2016), who develops
differentiability and stability properties, nor to Xia, Xie, Zhang
and Zhao (2020), who develop related well-posedness and
stochastic-flow results with diffusion-gradient exponents in
$[2,\infty)$.  Fedrizzi and Flandoli (2011, 2013) study
singular drifts with additive noise, proving pathwise uniqueness
and stochastic-flow regularity under assumptions with $p\ge2$.
Thus, unlike our zero-drift setting, the latter papers place the
low regularity in the drift rather than in the diffusion
coefficient. In one dimension, diffusion regularity can be transferred to drift regularity (and vice versa) using the Lamperti transformation $Y_t:=F(X_t)$ where $F'=\frac1\sigma$. Formally, we have 
\begin{align}\label{Lamperti_Dynamics}
    \dd Y_t=b(Y_t)\,\dd t+\dd B_t,
    \qquad
    b(y)
    :=
    -\frac12\sigma'\bigl(F^{-1}(y)\bigr).
\end{align}
When $\sigma$ is bounded above and bounded away from zero on a compact interval $J$, we have $
     b\in L^p(F(J))$ if and only if $ \sigma'\in L^p(J)$.  Therefore, in our setting, the $p\ge 2$ regularity in Fedrizzi and Flandoli (2011, 2013) corresponds to at least square integrability of $\sigma'$. In contrast, Theorem~5.1 of Ren and Zhang (2024) allows \(1<p<2\) and our cusp counterexample can be transferred by the Lamperti transformation to an additive-noise SDE satisfying Theorem~5.1's LPS assumptions to contradict their stated non-confluence conclusion.

Our motivation to write this paper was to understand better the grey area for
coefficients $\sigma\in W^{1,p}_{\mathrm{loc}}(\R)$ with
$p\in[1,2)$. In this range, $\lvert\sigma'\rvert^2 \in L^{\frac p2}_{\mathrm{loc}}$ and so the relevant Krylov estimates are
unavailable. Our results are:

\begin{enumerate}

\item In Section 2, our positive results provide new sufficient conditions for
strict comparison. In particular, they apply to some coefficients that belong to
$W^{1,p}_{\mathrm{loc}}(\R)$ for every $p\in(1,2)$ but not to
$W^{1,2}_{\mathrm{loc}}(\R)$, including
\begin{align}\label{ex_main}
    \sigma_\text{sym}(x):=1+\sqrt{|x|},
    \qquad
    \sigma_+(x):=1+\sqrt{x^+},
    \qquad x\in\R.
\end{align}

\item In Section \ref{sec:counterexample}, for each $p\in(1,2)$, we
construct a bounded, uniformly positive coefficient
$\sigma\in W^{1,p}_{\mathrm{loc}}(\mathbb R)$ for which global strong
existence and pathwise uniqueness hold  but strict comparison
fails. This contradicts the homeomorphism conclusion of
Theorem~1.1 of Zhang (2011) in the range $1<p<2$ and, after the
Lamperti transformation, the non-confluence conclusion of
Theorem~5.1 of Ren and Zhang (2024) in the range $1<p<2$. Our example
does not contradict these results outside the range $1<p<2$, nor
does it contradict any of the other papers discussed above.

\item In
Section~\ref{sec:critical-holder-counterexample}, for each $\beta\in(\frac12,1)$, we construct a bounded,
uniformly positive coefficient $\sigma\in C^{\beta}(\R)$ for which global
strong existence and pathwise uniqueness hold while strict comparison
fails. Consequently, $\sigma\in C^{\frac12}(\R)$ neither implies nor precludes strict comparison. 

\end{enumerate}

Finally, let us mention  Barlow (1982) and Barlow, Burdzy, Kaspi, and Mandelbaum (2001). Barlow (1982) constructs continuous, uniformly positive diffusion coefficients for which no
strong solution exists. In contrast, all our counterexamples have global strong and pathwise unique solutions. Barlow, Burdzy, Kaspi, and Mandelbaum (2001) construct a
non-degenerate coalescing example based on skew Brownian motion. By a scale transformation for skew Brownian motion
(see, e.g., Exercise~X.2.24 in Revuz and Yor, 1999),
this can equivalently be represented as an SDE with a discontinuous,
uniformly non-degenerate diffusion coefficient.
 In contrast, all our counterexamples have continuous
diffusion coefficients.

\section{Main positive result}\label{sec2}

\subsection{Setup and motivating examples}
On our probability space $(\Omega, \sF,\P)$, we assume $B = (B_t)_{t\ge0}$ is a one-dimensional standard Brownian motion. 
Throughout, we use the standard augmented Brownian filtration denoted by $\mathcal F_t$ for $t\ge0$, which satisfies the usual conditions.  All adapted processes and stopping times below are with respect to this filtration.  

For a solution of the SDE \eqref{eq:intro-sde} started from \(x\in I\), we let
\(\xi^x\) denote its maximal lifetime in the open interval \(I\).   If a
boundary extension --- such as absorption  ---  is imposed after the path leaves \(I\), that continuation is not
included in \(\xi^x\). This distinction is illustrated in the examples below.

\begin{assumption}\label{ass1}
Let \(I=(\ell,r)\subseteq\mathbb R\) be an open interval and let
\(\sigma:I\to(0,\infty)\) be continuous. Assume that, for all
\(x\in I\), the SDE~\eqref{eq:intro-sde} admits a strong solution up to
 \(\xi^x\), and that pathwise uniqueness holds
 up to  \(\xi^x\).
\hfill\(\diamondsuit\)
\end{assumption}
\noindent This assumption holds, e.g., for $\sigma \in C^{\frac12}_{\text{loc}}(I)$ because \eqref{LeGall1}  ensures pathwise uniqueness of the SDE \eqref{eq:intro-sde}. Local weak existence follows from the Engelbert--Schmidt conditions. The Yamada--Watanabe principle then yields strong existence.

\begin{example} \label{ex:1}
Consider \(I:=(0,\infty)\) and the Feller coefficient
\[
    \sigma(x):=\sqrt{x},\quad x\in (0,\infty).
\]
Condition \eqref{LeGall1} gives pathwise uniqueness.  The degeneracy $\sigma(0) = 0$ happens  outside $I=(0,\infty)$ and so $\sigma\in\operatorname{Lip}_{\mathrm{loc}}(I)$. Remark 2.1 in Yamada and Ogura (1981) ensures strict comparison in the sense
\[
\P\!\left(
   X_t^y>X_t^x\text{ for all }0\le t<\xi^x\wedge\xi^y
\right)=1,
\qquad y>x>0.
\]
Proposition 5.5.22(b) in Karatzas and Shreve (1991) ensures that $\lim_{t\uparrow \xi^z}X_t^z =0$ a.s. Consequently, the lifetime $\xi^z$  in the open interval $(0,\infty)$ equals
\[
    \tau_z:=\inf\{t\geq0:X_t^z=0\},\quad z>0.
\]

Alternatively, we can extend the path after $\tau_z$ by absorption at zero.  For this extension,  we have $\Pp(\tau_z<\infty)=1$ for $z>0$ because
\[
    \Pp(\tau_z\leq t)=\Pp(X_t^z=0)
    =\exp\!\left(-\frac{2z}{t}\right)\to1
    \quad \text{as } \quad t\to\infty.
\]
Because $\xi^z = \tau_z$, this gives $\P(\xi^z<\infty)=1$. The extended process is global but because every path started from $z>0$ hits zero almost surely, strict comparison for the absorbed extensions on the entire time interval $[0,\infty)$ fails.

\hfill\(\diamondsuit\)

\end{example}

\begin{example} \label{ex:2}
Consider \(I:=\R\) and the coefficients
\[
    \sigma_\text{sym}(x):=\sqrt{|x|},\quad     \sigma_+(x):=\sqrt{x^+},\quad x\in \R.
\]
Because \(\sigma_\text{sym}(0)=\sigma_+(0)=0\), these examples do not satisfy Assumption~\ref{ass1}, but \eqref{LeGall1} holds and gives pathwise uniqueness. Moreover, they belong to $W^{1,p}_\text{loc}(\R)$ for $1\le p<2$ but not to $W^{1,2}_\text{loc}(\R)$. Because every positive starting solution eventually hits zero and gets absorbed there, strict comparison fails. 
\hfill\(\diamondsuit\)
\end{example}

\begin{example} \label{ex:sqrt-3}
Consider \(I:=\R\) and the coefficients \eqref{ex_main}. The  function $\sigma_\text{sym}$ belongs to $C^\frac12_{\text{loc}}(\R)$ and so \eqref{LeGall1} gives pathwise uniqueness and  Assumption~\ref{ass1} holds.  As in Example  \ref{ex:2},  $\sigma_\text{sym}$ only belongs to $W^{1,p}_\text{loc}(\R)$ for $1\le p<2$ but not $W^{1,2}_\text{loc}(\R)$. For $\sigma_\text{sym}$,  \eqref{YOcond1} fails. Indeed, for all sufficiently small $h>0$, we have
\[
    |\sigma_\text{sym}(h)-\sigma_\text{sym}(0)|=\sqrt h,
\]
so every local modulus $\rho$ for $\sigma_\text{sym}$ must satisfy $\rho(h)\ge\sqrt h$. Hence, for sufficiently small $\varepsilon>0$,
\[
    \int_0^\varepsilon \frac{h\,\dd h}{\rho(h)^2}\le\varepsilon<\infty,
\]
and \eqref{YOcond1} fails. In Appendix \ref{Yamada_fails}, we show that \eqref{Yamada1986} cannot be applied either. Nevertheless, strict comparison holds for $\sigma_\text{sym}$ because of Theorem~\ref{thm:general-criterion} below. In Example~\ref{ex:sqrt-3B} below, we will verify the sufficient conditions of Theorem~\ref{thm:general-criterion}. The same analysis and conclusions hold for  $   \sigma_+$. 
\hfill\(\diamondsuit\)
\end{example}

\begin{example} \label{ex:new} For $\beta \in (0,\frac12)$, we define  following unbounded,
globally $\beta$-H\"older continuous functions
\begin{align}\label{eq:untruncated-beta}
\sigma_{+}(x)
    &:=1+(x^+)^\beta,\quad 
   \sigma_{\mathrm{sym}}(x)
    :=1+|x|^\beta,
    \quad x\in\mathbb R,
\end{align}
Because $\sigma_+'(x) = \beta x^{\beta-1}$ for $x>0$, we see $\sigma_+ \in W^{1,p}_\text{loc}(\R)$ for all $p\in [1,\frac{1}{1-\beta})$. To see that Assumption \ref{ass1} holds, we define
\[
    f(x):=(x^+)^{2\beta},\quad x\in \R.
\]
Then, for  $x\le y$, we have
\begin{align*}
    |\sigma_+(y)-\sigma_+(x)|^2=
    \left((y^+)^\beta-(x^+)^\beta\right)^2\le
    (y^+)^{2\beta}-(x^+)^{2\beta}=
    f(y)-f(x).
\end{align*}
Therefore, \eqref{NakaoLeGall} holds whereas \eqref{LeGall1} fails. In Section 3, we prove that strict comparison fails in the sense that
there exists $d_\beta>0$ such that the stopping times
\[
    \tau_d:=\inf\{t\ge0:X_t^{-\frac{d}2}=X_t^{\frac{d}2}\},\quad 0<d<d_\beta,
\]
satisfy $\Pp(\tau_d<\infty)>0$. The same analysis holds for $\sigma_{\text{sym}}$ defined in \eqref{eq:untruncated-beta}.
\hfill\(\diamondsuit\)

\end{example}

\subsection{Main result }
We define
\begin{align}
    \Sigma(M,D)
     &:=\frac{\sigma(M+\frac{D}2)+\sigma(M-\frac{D}2)}2,\quad M\in\R, \quad D\ge 0,
      \label{eq:Sigma-def}\\
    \Delta(M,D)
      &:=\sigma(M+\frac{D}2)-\sigma(M-\frac{D}2),\quad M\in\R, \quad D\ge 0,
      \label{eq:Delta-def}\\
    \Gamma(M,D)
      &:=\Sigma(M,D)-\frac MD\,\Delta(M,D),\quad M\in\R, \quad D> 0.
      \label{eq:Gamma-def}
\end{align}

\begin{theorem}\label{thm:general-criterion} Let Assumption~\ref{ass1} hold with $I:=\R$ and suppose that:
\begin{enumerate}[label=\textnormal{(\roman*)}]
    \item $\sigma$ is locally Lipschitz on \(\R\setminus\{0\}\).
\item There are $R>0$, a continuous nonnegative function
$Q:\mathbb{R}\to[0,\infty)$, and a constant $C_Q\in[0,\infty)$ such that
\begin{align}\label{eq:Q-decay}
\int_0^b Q(u)\,\dd u
\le C_Q\log(1+b),
\quad
\int_{-b}^0 Q(u)\,\dd u
\le C_Q\log(1+b),
\quad b>0,
\end{align}
and
\begin{align}\label{eq:direct-condition}
    \Delta(M,D)^2
    \le
    D\,Q\!\left(\frac{M}{D}\right)\Gamma(M,D)^2,
    \qquad |M|<R,\quad 0<D<R.
\end{align}

\end{enumerate}
Then, for all $x,y\in\R$ with $x<y$, we have
\begin{align}\label{strict0}
\P\!\left(X_t^x<X_t^y\text{ for all }t\ge0\right)=1.
\end{align}

\end{theorem}

\begin{proof} Since \(X^z\) in \eqref{eq:intro-sde} is a continuous local martingale, its scale function $p$ is affine with 
$p(-\infty)=-\infty$ and $p(+\infty)=+\infty$. Proposition~5.5.22(a) in Karatzas and Shreve (1991)
gives $\P(\xi^z=\infty)=1$ for all $z\in \R$. Fix $y>x$ and set
\begin{align}\label{MandD}
    M_t:=\frac{X_t^x+X_t^y}{2},
    \qquad
    D_t:=X_t^y-X_t^x.
\end{align}
We define the first meeting time (stopping time)
\begin{align}\label{firstmeetingtime}
    \tau:=\inf\{t\ge0:D_t=0\}.
\end{align}
Since $D_0=y-x>0$ and $D$ is continuous, $\tau>0$ almost surely. Path continuity gives $D_t>0$ on the stochastic interval  $t\in [0,\tau)$.  
\medskip

\noindent {\bf Step 1/4:} This step creates a lower bound for the Lyapunov function. For $0<t<\tau$, we have the dynamics
\begin{align}\label{dMdD}
\begin{split}
    \dd M_t
    &=\frac{\sigma(M_t+\frac{D_t}2)+\sigma(M_t-\frac{D_t}2)}2\,\dd B_t=\Sigma(M_t,D_t)\dd B_t, \\
    \dd D_t
    &=\big(\sigma(M_t+\tfrac{D_t}2)-\sigma(M_t-\tfrac{D_t}2)\big)\dd B_t = \Delta(M_t,D_t)\dd B_t.
\end{split}
\end{align}
We define $\varphi :\R \to\R$ by
\begin{equation}\label{eq:phi-streamlined}
    \varphi(z):=-\int_0^z (z-u)Q(u)\,\dd u,\quad z\in \R.
\end{equation}
Then, $\varphi \in C^2(\R)$ and satisfies
\[
    \varphi(0)=\varphi'(0)=0,
    \qquad
    \varphi''(z)=-Q(z)\le0.
\]
These properties ensure $\varphi$ is concave and $\varphi \le 0$. Our proof is based on the Lyapunov function
\begin{equation}\label{eq:H-streamlined}
    H(M,D):=-\log D+D\varphi\!\left(\frac MD\right),
\quad M\in \R,   \quad D>0,
\end{equation}
where we  note that $H(M,0)$ is undefined. For $ M\in \R$ and $ D>0$, differentiation of \eqref{eq:H-streamlined} gives
\begin{align*}
    H_{MM}(M,D)
      &=\frac1D\varphi''\!\left(\frac MD\right)
       =-\frac1D Q\!\left(\frac MD\right),\\
    H_{MD}(M,D)
      &=-\frac{M}{D^2}\varphi''\!\left(\frac MD\right)
       =\frac{M}{D^2}Q\!\left(\frac MD\right),\\
    H_{DD}(M,D)
      &=\frac1{D^2}+\frac{M^2}{D^3}\varphi''\!\left(\frac MD\right)
       =\frac1{D^2}-\frac{M^2}{D^3}Q\!\left(\frac MD\right).
\end{align*} 
 It\^o's formula produces the following drift of $H(M_t,D_t)$ 
\begin{align}\label{eq:H-drift-streamlined}
\frac{\Delta(M_t,D_t)^2}{2D_t^2}
-
\frac{Q(\frac {M_t}{D_t})}{2D_t}\Gamma(M_t,D_t)^2,\quad t\in [0,\tau).
\end{align}
Condition~\eqref{eq:direct-condition} implies that \eqref{eq:H-drift-streamlined} is nonpositive whenever
$|M_t|<R$ and $0<D_t<R$.  Thus, up to stopping inside $(-R,R)\times(0,R)$, the process $H(M_t,D_t)$ is a local supermartingale.

Next, we prove a lower bound for $H$. From \eqref{eq:Q-decay} and \eqref{eq:phi-streamlined}, we have
\[
\varphi(z)
\ge
-|z|\int_{\min(0,z)}^{\max(0,z)}Q(u)\,\dd u
\ge
-C_Q|z|\log(1+|z|),
\qquad z\in\R.
\]
Choose $\rho>0$ such that
\begin{equation}\label{eq:rho-streamlined}
    \rho<R,
    \qquad
    \rho\le1,
    \qquad
    C_Q\rho\le\frac12,
\end{equation}
where the third condition is unnecessary when $C_Q=0$.  For $|M|\le\rho$ and $0<D\le\rho$, we have
\begin{align*}
    H(M,D)
    &\ge \log\frac1D
       -C_Q|M|\log\!\left(1+\frac{|M|}{D}\right)\\
    &\ge \log\frac1D
       -C_Q|M|\log\frac2D\\
    &=\left(1-C_Q|M|\right)\log\frac1D-C_Q|M|\log2\\
    &\ge \frac12\log\frac1D-\frac12\log2,
\end{align*}
where the second inequality uses $D+|M| \le 2\rho \le2$. Consequently,
\begin{equation}\label{eq:H-blowup-streamlined}
   H(M,D)+\frac12\log2\ge -\frac12\log D,\quad |M|\le \rho,\quad 0<D\le \rho.
\end{equation}

\noindent {\bf Step 2/4:} This step provides the upper and lower bounds needed in the next step. We fix deterministic times $0\le q<T<\infty$, an integer $K\in\N$, and define the $\mathcal F_q$-measurable set
\begin{equation}\label{eq:E-streamlined}
    \mathcal E
    :=\big\{|M_q|<\rho,\quad 0<D_q<\rho,\quad H(M_q,D_q)\le K\big\},
\end{equation}
where $\rho>0$ satisfies \eqref{eq:rho-streamlined}. For $0<\varepsilon<\rho$, we define the bounded stopping times
\begin{equation}\label{eq:theta-streamlined}
    \theta_\varepsilon
    :=
      \inf\big\{t\ge q:
      |M_t|\ge\rho\ \text{or}\ D_t\le\varepsilon\ \text{or}\ D_t\ge\rho
      \big\}\land T\in [q,T].
\end{equation}
For $\omega \in \mathcal E\cap\{D_q>\varepsilon\}$, we have $q<\theta_\varepsilon(\omega)$ and so
\begin{align}\label{Fatou1}
    |M_t(\omega)|\le\rho,
    \qquad
    \varepsilon\le D_t(\omega)\le\rho,\quad t\in [q,\theta_\varepsilon(\omega)].
\end{align}
On the compact region $M\in[-\rho,\rho]$ and
$D\in[\varepsilon,\rho]$, the function $H$ and all its first and
second derivatives are bounded.  Moreover, $M\pm \frac{D}2$ remain in a
fixed compact interval.  By continuity of $\sigma$, both diffusion
coefficients $\Sigma(M,D)$ and $\Delta(M,D)$ in \eqref{dMdD} are bounded on this region. Since \(\mathcal E\cap\{D_q>\varepsilon\}\in\mathcal F_q\), we can multiply by this set's indicator to see that 
the stochastic integral produced by applying It\^o's
formula from
$q$ to $\theta_\varepsilon$ to compute the dynamics of 
$H(M_{t\land\theta_\varepsilon},D_{t\land\theta_\varepsilon})$ 
is a square integrable martingale with zero expectation.  Since the drift in
\eqref{eq:H-drift-streamlined} is nonpositive, we take expectations 
to get
\begin{align}\label{eq:stopped-expectation-streamlined}
    \E\left[
       \mathbf 1_{\mathcal E\cap\{D_q>\varepsilon\}}
       H(M_{\theta_\varepsilon},D_{\theta_\varepsilon})
    \right]
    \le
    \E\left[
       \mathbf 1_{\mathcal E\cap\{D_q>\varepsilon\}}H(M_q,D_q)
    \right].
\end{align}
On  $\mathcal E\cap\{D_q\le\varepsilon\}$, we
have $\theta_\varepsilon=q$, and therefore
\begin{align}\label{eq:stopped-expectation-streamlined0}
    \E\left[
       \mathbf 1_{\mathcal E\cap\{D_q\le\varepsilon\}}
       H(M_{\theta_\varepsilon},D_{\theta_\varepsilon})
    \right]
    =
    \E\left[
       \mathbf 1_{\mathcal E\cap\{D_q\le\varepsilon\}}H(M_q,D_q)
    \right].
\end{align}
Adding \eqref{eq:stopped-expectation-streamlined} and
\eqref{eq:stopped-expectation-streamlined0} yields
\begin{equation}\label{eq:stopped-expectation-streamlined2}
    \E\left[
       \mathbf 1_{\mathcal E}
       H(M_{\theta_\varepsilon},D_{\theta_\varepsilon})
    \right]
    \le
    \E\left[
       \mathbf 1_{\mathcal E}H(M_q,D_q)
    \right]
    \le K\P(\mathcal E).
\end{equation}

On $\mathcal E\cap\{D_q>\varepsilon\}$, we use  \eqref{Fatou1} for $t=\theta_\varepsilon$ to see
$\varepsilon\le D_{\theta_\varepsilon}\le\rho$. On
$\mathcal E\cap\{D_q\le\varepsilon\}$, we have
$\theta_\varepsilon=q$ and hence
$0<D_{\theta_\varepsilon}=D_q\le\varepsilon$.  All in all, $D_{\theta_\varepsilon}>0$ on 
$\mathcal E$ and the lower bound in \eqref{eq:H-blowup-streamlined} is applicable and gives
\begin{equation}\label{eq:uniform-lower-fatou}
    \mathbf 1_{\mathcal E}
    \left(
       H(M_{\theta_\varepsilon},D_{\theta_\varepsilon})
       +\frac12\log2
    \right)
    \ge  -\frac12\log (D_{\theta_\varepsilon})\ge0,\qquad 0<\varepsilon<\rho.
\end{equation}
The inequality in \eqref{eq:stopped-expectation-streamlined2} gives the uniform expectation
bound
\begin{equation}\label{eq:uniform-expectation-fatou}
\E\left[
    \mathbf 1_{\mathcal E}
    \left(
       H(M_{\theta_\varepsilon},D_{\theta_\varepsilon})
       +\frac12\log2
    \right)
\right]
\le
\left(K+\frac12\log2\right)\P(\mathcal E),
\qquad 0<\varepsilon<\rho.
\end{equation}

\noindent{\bf Step 3/4:} This step rules out $X^x$ and $X^y$ meeting at the cusp point 0. We use $\tau$ from \eqref{firstmeetingtime} to define the set
\begin{align}\label{setA}
    \mathcal A
    :=\mathcal E\cap
    \left\{
      q<\tau\le T,\quad
      |M_t|<\rho,\quad 0<D_t<\rho
      \text{ for all }q\le t<\tau
    \right\}.
\end{align}
Let $\varepsilon_n\downarrow0$ with $0<\varepsilon_n<\rho$.
By \eqref{eq:uniform-lower-fatou} and
$\mathcal A\subseteq\mathcal E$, for all $n$, we have
\[
    \mathbf 1_{\mathcal A}
    \left(
       H(M_{\theta_{\varepsilon_n}},D_{\theta_{\varepsilon_n}})
       +\frac12\log2
    \right)
    \ge0.
\]
This uniform lower bound allows us to use Fatou's lemma below.

Fix $\omega\in\mathcal A$ so in particular $\tau(\omega) \le T <\infty$.  Since $D_q(\omega)>0$, there exists an
$\omega$-dependent integer $N(\omega)$ such that
$\varepsilon_n<D_q(\omega)$ for all $n\ge N(\omega)$.  The path
$t\mapsto D_t(\omega)$ is continuous, and is strictly positive on
$[q,\tau(\omega))$, and satisfies $D_{\tau}(\omega)=0$.
Consequently, for all $n\ge N(\omega)$, the path $D(\omega)$ hits 
$\varepsilon_n$ strictly before $\tau(\omega)$.  By the definition of
$\mathcal A$,
\[
    |M_t(\omega)|<\rho
    \qquad\text{and}\qquad 0<
    D_t(\omega)<\rho,
    \qquad q\le t<\tau(\omega),
\]
so neither of the stopping conditions
$|M_t|\ge\rho$ or $D_t\ge\rho$ can occur before $\theta_{\varepsilon_n}(\omega)$.
Since $\theta_{\varepsilon_n}(\omega)<\tau(\omega)\le T$, the
truncation at $T$ cannot occur first either.  Therefore, for all
$n\ge N(\omega)$, $\theta_{\varepsilon_n}(\omega)$ is the
first time after $q$ at which $D(\omega)$ hits $\varepsilon_n$.  Hence, eventually, we have
\[
    D_{\theta_{\varepsilon_n}}=\varepsilon_n,
    \qquad
    |M_{\theta_{\varepsilon_n}}|<\rho
    \qquad\text{on }\mathcal A.
\]

The lower bound \eqref{eq:H-blowup-streamlined} gives, on $\mathcal A$,
\[
    H(M_{\theta_{\varepsilon_n}},D_{\theta_{\varepsilon_n}})
    +\frac12\log2
    \ge
    \frac12\log\frac1{\varepsilon_n}
    \longrightarrow \infty\quad \text{as} \quad n\to\infty.
\]
The lower bound \eqref{eq:uniform-lower-fatou} allows us to use Fatou's lemma to see
\begin{align*}
&\E\Bigg[
\mathbf 1_{\mathcal A}
\liminf_{n\to\infty}
\left(
H(M_{\theta_{\varepsilon_n}},D_{\theta_{\varepsilon_n}})
+\frac12\log2
\right)
\Bigg]
\\
&\qquad\le
\liminf_{n\to\infty}
\E\left[
\mathbf 1_{\mathcal A}
\left(
H(M_{\theta_{\varepsilon_n}},D_{\theta_{\varepsilon_n}})
+\frac12\log2
\right)
\right]
\\
&\qquad\le
\liminf_{n\to\infty}
\E\left[
\mathbf 1_{\mathcal E}
\left(
H(M_{\theta_{\varepsilon_n}},D_{\theta_{\varepsilon_n}})
+\frac12\log2
\right)
\right]
\\
&\qquad\le
\left(K+\frac12\log2\right)\P(\mathcal E)<\infty,
\end{align*}
where the last inequality uses \eqref{eq:uniform-expectation-fatou}. Because the random variable inside the first expectation is $+\infty$ on
$\mathcal A$, we get $\P(\mathcal A)=0$.

On the set $(X_\tau^x=X_\tau^y=0,\ \tau<\infty)$, we have
$M_\tau=0$ and $D_\tau=0$.  By continuity, for each sample path in
this set, there exists $\delta (\omega)=\delta\in(0,\tau)$ such
that
\[
    |M_t|<\rho
    \qquad\text{and}\qquad
    0<D_t<\rho,
    \qquad t\in(\tau-\delta,\tau).
\]
Hence, for each such sample path, there exists a rational number
$q\in\mathbb Q\cap(\tau-\delta,\tau)$ and an integer $T$ such that
$\tau\le T$.  Since $q<\tau$, we have $D_q>0$, and therefore
$H(M_q,D_q)$ is finite.  Thus, there exists an integer $K$ such that
$H(M_q,D_q)\le K$.  Consequently,
\[
    (X_\tau^x=X_\tau^y=0,\ \tau<\infty)
    \subseteq
    \bigcup_{q\in\mathbb Q_+}
    \bigcup_{T\in\mathbb N}
    \bigcup_{K\in\mathbb N}
    \mathcal A_{q,T,K},
\]
where $\mathcal A_{q,T,K}$ denotes the set $\mathcal A$ defined in \eqref{setA} 
with the corresponding choices of $q,T$, and $K$.  Each
$\mathcal A_{q,T,K}$ has probability zero by the preceding argument,
and the union is countable.  Hence,
\[
    \P(X_\tau^x=X_\tau^y=0,\ \tau<\infty)=0.
\]

\noindent {\bf Step 4/4:} It remains to exclude $X^x$ and $X^y$ meeting away from the cusp 0. Because $\sigma$ is locally Lipschitz away from zero, the argument is standard and is omitted.  \end{proof}

%
%
%

The pointwise condition
\begin{align}\label{pointcond}
    Q(s)\le \frac{C_Q}{1+|s|},
    \qquad s\in\mathbb{R},
\end{align}
implies \eqref{eq:Q-decay}. Indeed, for $z\ge0$, we have
\[
\int_0^z Q(u)\,\dd u
\le
C_Q\int_0^z\frac{\dd u}{1+u}
=
C_Q\log(1+z).
\]
For $z\le0$, the substitution $v=-u$ gives the bound.

\begin{example}\label{ex:sqrt-3B}\label{ex:sqrt-positive-part}
(Example~\ref{ex:sqrt-3} continued) Consider the two coefficients
$\sigma_+$ and $\sigma_{\mathrm{sym}}$ defined in \eqref{ex_main}. Both $\sigma_+$ and $\sigma_{\mathrm{sym}}$ satisfy Assumption~\ref{ass1} and are locally Lipschitz on
\(\mathbb R\setminus\{0\}\). In the following, we will verify the sufficient conditions of
Theorem~\ref{thm:general-criterion} for
\(\sigma\in\{\sigma_+,\sigma_{\mathrm{sym}}\}\) simultaneously. For \(z\in\mathbb R\), we define \(A(z)\) and \(B(z)\) by
\[
\begin{array}{c|cc}
 & A(z) & B(z) \\ \hline
\sigma=\sigma_+
 & \sqrt{(z+\frac12)^+} & \sqrt{(z-\frac12)^+} \\[1mm]
\sigma=\sigma_{\mathrm{sym}}
 & \sqrt{|z+\frac12|} & \sqrt{|z-\frac12|}
\end{array}
\]
and we define
\begin{align}\label{eq:sqrt-h-g}
    h(z)&:=A(z)-B(z),
    &
    g(z)&:=\frac{A(z)+B(z)}2-z h(z),
    &
    Q(z)&:=h(z)^2.
\end{align}
The function \(Q\) is continuous and nonnegative. Because
\[
    |\sqrt a-\sqrt b|^2\le |a-b|,
    \qquad a,b\ge0,
\]
and both $x^+$ and $|x|$ are 
Lipschitz with constant 1, we have
\[
Q(z)= h(z)^2 \le1,\quad z\in\R,   \qquad Q(z) \le\frac{2}{1+|z|},    \quad |z|\le1.
\]
For \(z\ge1\), the functions $A(z)$ and $B(z)$ are the same for both $\sigma_+$ and $\sigma_{\mathrm{sym}}$ and produce
\[
    h(z)
    =
    \sqrt{z+\frac12}-\sqrt{z-\frac12}
    =
    \frac{1}{
       \sqrt{z+\frac12}+\sqrt{z-\frac12}},\quad z\ge 1.
\]
Therefore, we have
\[
    Q(z) = h(z)^2
    \le
    \frac{1}{z+\frac12}
    \le
    \frac{2}{1+z},\quad z\ge1.
\]
For \(\sigma=\sigma_+\), we have \(Q(z)=0\) when \(z\le-1\),
whereas for \(\sigma=\sigma_{\mathrm{sym}}\), the function \(Q\) is
even. Consequently, in both cases,  the stronger pointwise condition~\eqref{pointcond} holds with
\(C_Q=2\).

Next, for \(M\in\mathbb R\) and \(D>0\), and set $z:=\frac MD\in \R$. For both coefficients \(\sigma\in\{\sigma_+,\sigma_{\mathrm{sym}}\}\), we have
\[
    \sigma\!\left(M+\frac D2\right)=1+\sqrt D\,A(z),
    \qquad
    \sigma\!\left(M-\frac D2\right)=1+\sqrt D\,B(z).
\]
Therefore, the functions \(\Delta\) and \(\Gamma\) from
\eqref{eq:Delta-def}--\eqref{eq:Gamma-def} satisfy
\begin{align}\label{eq:sqrt-profiles}
    \Delta(M,D)&=\sqrt D\,h(z),\quad     \Gamma(M,D)=1+\sqrt D\,g(z).
\end{align}
In particular,
\begin{equation}\label{eq:sqrt-Delta-Q}
    \Delta(M,D)^2=D h(z)^2 = DQ(z).
\end{equation}

It remains to show \eqref{eq:direct-condition}, which is implied by \(g\ge0\). We set
\[
    a:=z+\frac12,
    \qquad
    b:=z-\frac12,\quad z>\frac12.
\]
For  both $\sigma_+$ and $\sigma_{\mathrm{sym}}$, we have
\[
\begin{aligned}
    g(z)
    &=a\sqrt b-b\sqrt a=\sqrt{ab}\,(\sqrt a-\sqrt b)\ge0,\quad z>\frac12.
\end{aligned}
\]
For \(\sigma=\sigma_+\), we have 
\begin{align*}
g(z) = 
\begin{cases}
0,\quad z \in (-\infty,-\frac12),\\
    \left(\frac12-z\right)\sqrt{z+\frac12},\quad   z\in  [-\frac12,\frac12],
\end{cases}
\end{align*}
which is nonnegative. For \(\sigma=\sigma_{\mathrm{sym}}\), the function \(g\) is even.
For \(0\le z\le\frac12\), we set
\[
    a:=z+\frac12,
    \qquad
    b:=\frac12-z,
\]
to get 
\[
    g(z)=b\sqrt a+a\sqrt b\ge0,\quad z \in [0,\frac12].
\]
To see that \eqref{eq:direct-condition} holds, we insert $g\ge0$ into \eqref{eq:sqrt-profiles} to see 
\[
    \Gamma(M,D)\ge1,
    \qquad M\in\mathbb R,\quad D>0.
\]
Combining $\Gamma\ge1$ with \eqref{eq:sqrt-Delta-Q} yields \eqref{eq:direct-condition} because
\[
    DQ\!\left(\frac MD\right)\Gamma(M,D)^2
    \ge
    DQ\!\left(\frac MD\right)
    =
    \Delta(M,D)^2,\quad M\in\mathbb R,\quad D>0.
\]
\hfill\(\diamondsuit\)
\end{example}

\subsection{Extension to countable cusps}
In this section we briefly describe an extension to countably many cusps. We say that a coefficient $\sigma$ is \emph{countably piecewise locally Lipschitz}
on $I=(\ell,r)$ if and only if there exists a strictly increasing sequence of \emph{cusps}
$(a_n)_{n\in\mathbb Z}\subset I$ such that
\[
    \lim_{n\to-\infty}a_n=\ell,
    \qquad
    \lim_{n\to\infty}a_n=r,
\]
and
\[
    \sigma\big|_{(a_n,a_{n+1})}
    \in \operatorname{Lip}_{\mathrm{loc}}(a_n,a_{n+1}),
    \qquad n\in\mathbb Z.
\]

We adjust  \eqref{eq:Sigma-def}-\eqref{eq:Gamma-def} to 
      \begin{align*}
    \Sigma_n(M,D)
      &:=\frac{\sigma(a_n+M+\frac{D}2)+\sigma(a_n+M-\frac{D}2)}2,\\
    \Delta_n(M,D)
      &:=\sigma(a_n+M+\frac{D}2)-\sigma(a_n+M-\frac{D}2),\\
    \Gamma_n(M,D)
      &:=\Sigma_n(M,D)-\frac MD\Delta_n(M,D),
\end{align*}
where 
 $D>0$ and $M\in\mathbb R$ are such that
$a_n+M+\frac D2,\ a_n+M-\frac D2\in I$ (for $\Sigma_n$ and $\Delta_n$, we also allow $D=0$). In the following result, no uniformity in $R_n$, $Q_n$, or $C_n$ is required. While Proposition~5.5.22 in Karatzas and Shreve (1991) rules out oscillatory behavior for $X^z$ in finite time,  $\P(\xi^z=\infty)<1$ is possible when $I$ has at least one finite endpoint (see, e.g.,  Example \ref{ex:1}).  So, unlike \eqref{strict0}, we need to cap with  $\xi^z$  in \eqref{strict1}  below.

\begin{corollary}\label{thm:countably-many-cusps} In addition to Assumption~\ref{ass1}, assume that $\sigma$ is countably piecewise locally Lipschitz.
For all $n\in\mathbb Z$, assume there are: a constant $R_n>0$, a continuous
nonnegative function $Q_n:\mathbb R\to[0,\infty)$, and a constant
$C_n<\infty$ such that
\[
    \frac32 R_n
    <\min\{a_n-a_{n-1},\,a_{n+1}-a_n\},
\]
\begin{align}
    \int_0^b Q_n(u)\,\dd u
      &\le C_n\log(1+b),
      &
    \int_{-b}^0 Q_n(u)\,\dd u
      &\le C_n\log(1+b),
      \qquad b>0,
      \label{eq:Qn-growth}
\end{align}
and
\begin{equation}
    \Delta_n(M,D)^2
    \le
    D\,Q_n\!\left(\frac MD\right)\Gamma_n(M,D)^2,
    \qquad |M|<R_n,\quad 0<D<R_n.
    \label{eq:direct-condition-n}
\end{equation}
Then, for all $x<y$ in $I$,
\begin{align}\label{strict1}
    \P\!\left(
       X_t^x<X_t^y
       \text{ for all }0\le t<\xi^x\wedge\xi^y
    \right)=1.
\end{align}
\end{corollary}

\begin{proof}
Collision away from the cusps $(a_n)_{n\in\mathbb Z}$ is excluded by the
local Lipschitz argument. Collision at a cusp $a_n$ is handled as in
the proof of Theorem~\ref{thm:general-criterion}. For brevity, we omit
the details.
\end{proof}

\section{Insufficiency of $W^{1,p}_\text{loc}$ for  $p\in(1,2)$}\label{sec:counterexample}

\subsection{Two counterexamples}
For a constant $\beta>0$, we define the bounded coefficients
\begin{align}\label{both_trunc}
    \sigma_+(x):=1+(x^+\wedge2)^\beta,
    \qquad
    \sigma_{\mathrm{sym}}(x):=1+(|x|\wedge2)^\beta,\quad x\in \R.
\end{align}
The following theorem shows that strict comparison fails for these coefficients and it has the following consequence for $p\in(1,2)$. For $0<\beta<\frac12$, we have 
\begin{align}\label{W1pmembership}
    \sigma_+,\sigma_{\mathrm{sym}}\in W^{1,p}_{\mathrm{loc}}(\R)
    \quad\Longleftrightarrow\quad
    1\le p<\frac1{1-\beta},
\end{align}
we can choose $\beta \in (1-\frac1p,\frac12)$ to produce a bounded, uniformly positive
 coefficient in $W^{1,p}_{\mathrm{loc}}(\R)$ for which  pathwise uniqueness and global strong existence hold but for which  strict comparison fails.

 When $\beta>\frac12$,  both untruncated 
coefficients in  \eqref{eq:untruncated-beta}  belong to $ W^{1,2}_{\mathrm{loc}}(\mathbb R)$ and so strict comparison holds by \eqref{Yamada1986}. The case $\beta=\frac12$ is covered by
Example~\ref{ex:sqrt-positive-part}, whereas Theorem~\ref{thm:main} produces
failure of strict comparison for $\beta\in(0,\frac12)$. All in all,  for the two formulas in  \eqref{eq:untruncated-beta} with $\beta>0$,
strict comparison holds if and only if $\beta\ge\frac12$.

\begin{theorem}\label{thm:main} Let  $0<\beta<\frac12$. 
For both coefficients in \eqref{both_trunc}, Assumption \ref{ass1} holds for $I:=\R$.
Nevertheless, there is a constant $d_{\beta}>0$ such that
\[
\forall d \in (0,d_{\beta}):\quad     \P\!\left(
       X_t^{-\frac{d}2}=X_t^{\frac{d}2}
       \text{ for some }t<\infty
    \right)>0.
\]
Consequently, strict comparison fails.  Furthermore, Assumption \ref{ass1} holds for $I:=\R$ for the untruncated coefficients \eqref{eq:untruncated-beta} but strict comparison again fails.
\end{theorem}

\begin{proof} 

\noindent{\bf Step 1/5:} For $\sigma_+$ in \eqref{both_trunc}, we define
\[
    f(x):=(x^+\wedge2)^{2\beta},\quad x\in\R.
\]
For $\sigma_{\mathrm{sym}}$ in \eqref{both_trunc}, we define
\[
    f(x):=\operatorname{sgn}(x)(|x|\wedge2)^{2\beta},\quad x\in\R\setminus \{0\}, 
    \qquad f(0):=0.
\]
In either case, $f$ is increasing and continuous. The inequality
\[
    (a-b)^2\le a^2-b^2,
    \qquad a\ge b\ge0,
\]
gives for $0\le x<y$ and $b:=(|x|\wedge2)^\beta$ and $a:=(|y|\wedge2)^\beta$ the inequality
\[
    \big(\sigma_{\mathrm{sym}}(x)-\sigma_{\mathrm{sym}}(y)\big)^2\le f(y)-f(x).
\]
A similar argument works for  $x<y\le0$. For $x<0<y$, we have for $b:=(|x|\wedge2)^\beta$ and $a:=(|y|\wedge2)^\beta$ the inequality
\[
 \big(\sigma_{\mathrm{sym}}(x)-\sigma_{\mathrm{sym}}(y)\big)^2=(b-a)^2\le a^2+b^2=f(y)-f(x).
\]
Because the argument for $\sigma_+$ is similar, we omit it. All in all, \eqref{NakaoLeGall} holds and gives pathwise uniqueness.  Weak existence and the
Yamada--Watanabe principle give a pathwise unique strong solution up to its
lifetime $\xi^z$. As already explained in the beginning of the proof of Theorem \ref{thm:general-criterion}, the strong solution is global (i.e., $\xi^z =\infty$ for all $z\in \R$).

For both coefficients $\sigma_+$ and $\sigma_\text{sym}$, the weak derivative has its only unbounded singularity at zero, which is comparable to $|x|^{\beta-1}$ as $x\to 0$ for $\sigma_\text{sym}$ and as $x\downarrow 0$ for $\sigma_+$.  Therefore, we have
\[
    \sigma' \in L^p_{\mathrm{loc}}(\R)
    \quad\Longleftrightarrow\quad
    p(\beta-1)>-1.
\]

\noindent{\bf Step 2/5:} Fix $d\in(0,1)$ and define
\[
    M_t:=\frac{X_t^{\frac{d}2}+X_t^{-\frac{d}2}}2,
    \qquad
    D_t:=X_t^{\frac{d}2}-X_t^{-\frac{d}2},\quad t\ge0.
\]
Comparison ensures that $D$ is a nonnegative continuous local
martingale.  Pathwise uniqueness ensures that $D$ remains at zero after
it first hits zero.  Define the stopping times
\begin{align}\label{taurho}
    \tau:=\inf\{t\ge0:D_t=0\},
    \qquad
    \rho:=\inf\{t\ge0:|M_t|\ge1\text{ or }D_t\ge1\},
\end{align}
and the open rectangle 
\[
    \mathcal R:=\{(M,D)\in\R\times(0,\infty): |M|<1,\ 0<D<1\}.
\]
Because $(M_0,D_0) \in     \mathcal R$, we have $(M_t,D_t)\in\mathcal R$ and
\[
    \left|M_t\pm\frac{D_t}{2}\right|
    \le |M_t|+\frac{D_t}{2}<\frac32,\quad t<\tau\wedge\rho.
\]
This estimate makes the truncations  in \eqref{both_trunc} inactive because both 
\[
    X_t^{-\frac{d}2}=M_t-\frac{D_t}{2},
    \qquad
    X_t^{\frac{d}2}=M_t+\frac{D_t}{2},\quad t<\tau\wedge\rho,
\]
belong to $(-\frac32,\frac32)$.  For $z\in\R$, we define $A(z)$ and $B(z)$ by
\[
\begin{array}{c|cc}
 & A(z) & B(z) \\ \hline
\sigma=\sigma_+ & (z+\frac12)_+^\beta & (z-\frac12)_+^\beta \\[1mm]
\sigma=\sigma_{\mathrm{sym}}
 & |z+\frac12|^\beta & |z-\frac12|^\beta
\end{array}
\]
and --- as in \eqref{eq:sqrt-h-g} --- define 
\begin{align}\label{h_g}
    h(z):=A(z)-B(z),
    \qquad
    g(z):=\frac{A(z)+B(z)}2-z h(z),\quad z\in\R.
\end{align}
For $D>0$ and $M\in\R$, we have the representations
\begin{align*}
 \sigma_{\mathrm{sym}}\!\left(M\pm\frac D2\right)
 &=1+D^\beta
 \left(
   \left|\frac MD\pm\frac12\right|\wedge\frac2D
 \right)^\beta,\\
 \sigma_+\!\left(M\pm\frac D2\right)
 &=1+D^\beta
 \left(
   \left(\frac MD\pm\frac12\right)^+\wedge\frac2D
 \right)^\beta.
\end{align*}
The truncations are inactive when
\[
\sigma_{\mathrm{sym}}:    \left|M\pm\frac D2\right|\le2\quad \text{ and }\quad \sigma_+:   \left(M\pm\frac D2\right)^+\le2.
\]
 In particular, the common sufficient condition $|M\pm\frac D2|<2$ holds throughout $\mathcal R$.  Therefore, for  $(M,D) \in \mathcal R$, the functions $\Sigma$, $\Delta$, and $\Gamma$
from \eqref{eq:Sigma-def}--\eqref{eq:Gamma-def} become

\begin{equation}\label{eq:counter-profiles}
\begin{aligned}
    \Sigma(M,D)&=1+\frac{D^\beta}{2}
       \Big(A\big(\frac MD\big)+B\big(\frac MD\big)\Big),\\
    \Delta(M,D)&=D^\beta h\big(\frac MD\big),\\
    \Gamma(M,D)&=1+D^\beta g\big(\frac MD\big).
\end{aligned}
\end{equation}
Below, we will use \eqref{eq:counter-profiles} only on $\mathcal R$ and only for $(M_t,D_t)$ for times $t<\tau\wedge\rho$.

The rest of this step derives bounds for $h$ and $g$ defined in \eqref{h_g}. First, we consider $h$, which is trivially continuous.  For $\sigma_+$, $h(x)=0$ for 
$x\in (-\infty,-\frac12]$. For $\sigma_{\mathrm{sym}}$, $h$ is odd.  Therefore, to prove $h\in L^2(\R)$, it suffices to show that $h$ is square integrable at $+\infty$.  Both $\sigma_+$ and $\sigma_\text{sym}$ give 
\[
    h(z)
    =
    \left(z+\frac12\right)^\beta
    -
    \left(z-\frac12\right)^\beta>0, \quad z\ge1.
\]
The Mean-Value Theorem applied to $u\mapsto u^\beta$ produces
$\xi_z\in(z-\frac12,z+\frac12)$ such that
\[
    h(z)=\beta \xi_z^{\beta-1},\quad z\ge1.
\]
Because $\beta<1$ the function $u\to u^{\beta-1}$ is decreasing  and so the lower bound
\[
    \xi_z> z-\frac12\ge \frac z2, \quad z\ge1,
\]
gives the upper bound
\begin{align}\label{h_growth}
  0< h(z)
    \le
    \beta\left(\frac z2\right)^{\beta-1},
    \qquad z\ge1.
\end{align}
Because $\beta<\frac12$, $z\mapsto z^{2(\beta-1)}$ is integrable as $z\to\infty$, hence, \eqref{h_growth} gives $h\in L^2(\R)$.

Next, we consider $g$ in \eqref{h_g}. $g$ is trivially continuous too. We will show $g \ge0$. Both coefficients $\sigma_+$ and $\sigma_\text{sym}$ satisfy for
$a:=z+\frac12$ and $b:=z-\frac12$,
\begin{align}\label{g_property}
    g(z)=a b^\beta-b a^\beta
        =ab\bigl(b^{\beta-1}-a^{\beta-1}\bigr)\ge0,\quad z>\frac12.
\end{align}
For $\sigma=\sigma_+$, we have $g(z)=0$ when $z\le-\frac12$, and
\[
    g(z)=\left(\frac12-z\right)
          \left(z+\frac12\right)^\beta\ge0,
    \qquad z\in[-\frac12,\frac12].
\]
For $\sigma=\sigma_{\mathrm{sym}}$, the function $g$ is even and satisfies for $a:=z+\frac12$ and $b:=\frac12-z$, 
\[
    g(z)=b a^\beta+a b^\beta\ge0,\quad z\in[0,\frac12].
\]
We need the following upper bound for $g$. For $z\ge1$, we set $a:=z+\frac12$ and $b:=z-\frac12$. For both coefficients $\sigma_+$ and $\sigma_\text{sym}$, we have \eqref{g_property}. The Mean-Value Theorem applied to the decreasing function $u\mapsto u^{\beta-1}$ produces $\xi_z\in(b,a)$ such that
\[
\begin{aligned}
g(z) =ab\bigl(b^{\beta-1}-a^{\beta-1}\bigr)=
   ab (1-\beta)\xi_z^{\beta-2}(a-b)    =
   ab (1-\beta)\xi_z^{\beta-2}.
\end{aligned}
\]
Since $\beta<2$ and $\xi_z\ge b$, this gives
\[
    g(z)
    \le
    (1-\beta)a b^{\beta-1}\le     (1-\beta)\frac32 z  \left(\frac z2\right)^{\beta-1},
\]
where the last inequality uses 
\[
    a\le\frac{3z}2,
    \qquad
    b\ge\frac z2 ,\quad z\ge1.
\]
For $\sigma=\sigma_+$, we have $g(z)=0$ for $z\le-\frac12$. For 
$\sigma=\sigma_{\mathrm{sym}}$, the function $g$ is even.  Together with
boundedness on compact intervals, this gives an irrelevant constant $C>0$ such that ($C$ will change throughout the proof)
\begin{align}\label{gbound2}
    0\le g(z)\le C(1+|z|^\beta),
    \qquad z\in\R.
\end{align}

The bound in \eqref{gbound2} allows us to bound $\Gamma$ from \eqref{eq:counter-profiles} on 
 $ \mathcal R$ as follows. 
For $|M|\le1$ and $0<D\le1$, the bounds $g\ge0$ and  \eqref{gbound2} produce
\[
\begin{aligned}
1\le\Gamma(M,D)    &= 1 + D^\beta g(\frac MD)\\
    &\le 1+    C D^\beta    \left(        1+\left|\frac MD\right|^\beta    \right)\\
    &= 1+     C\bigl(D^\beta+|M|^\beta\bigr)\\
    &\le 1+ 2C.
\end{aligned}
\]
Consequently, after increasing $C$ if necessary, we have
\begin{equation}\label{eq:counter-Gamma-bound}
    1\le\Gamma(M,D)\le C,
    \qquad |M|\le1,
    \quad 0<D\le1.
\end{equation}

\noindent{\bf Step 3/5:} We define
\begin{equation}\label{eq:counter-K}
\begin{aligned}
    &K(M,D):=
    -D\int_0^{\frac MD}
       \left(\frac{M}{D}-u\right)h(u)^2\,\dd u\\
    &+\left(\int_0^\infty h(u)^2\,\dd u\right)M^+
     +\left(\int_{-\infty}^0 h(u)^2\,\dd u\right)M^-,\quad M\in\R,\quad D>0, 
\end{aligned}
\end{equation}
where the second line uses $h\in L^2(\R)$ from Step 2. The function $K$ is nonnegative because all 4 terms on the right-hand side of
\begin{align}\label{K_rep1}
    \frac{K(M,D)}{D}
    &=
    \begin{cases}
    \int_0^{\frac MD} u h(u)^2\,\dd u
    +
    \frac{M}{D}
    \int_{\frac MD}^\infty h(u)^2\,\dd u, \quad M\ge0,\quad D>0,\\
    -\frac{M}{D}
    \int_{-\infty}^{\frac MD} h(u)^2\,\dd u
    -
    \int_{\frac MD}^0 u h(u)^2\,\dd u, \quad M\le0,\quad D>0,
    \end{cases}
\end{align}
are nonnegative. We need the following upper bound on $K$.  For $\sigma_+$, we have $h(z)=0$ for $z\le-\frac12$, while for
$\sigma_{\mathrm{sym}}$, $h$ is odd. Hence, \eqref{h_growth} implies
\begin{align}\label{hsquared}
    h(z)^2\le C|z|^{2\beta-2},
    \qquad |z|\ge1.
\end{align}
For $D>0$, we set $z:= \frac M D$. When  $M\ge0$ and $0\le z\le1$, \eqref{K_rep1} ensures that $ \frac{K(M,D)}D$ is uniformly bounded (uses  $h\in L^2(\R)$). When  $M\ge0$ and $z\ge1$,  \eqref{K_rep1} and \eqref{hsquared} produce the bound
\begin{align}\label{new_K_bound}
\begin{split}
    \frac{K(M,D)}D
    &=
    \int_0^{z}u h(u)^2\,\dd u
    +
    z\int_{z}^{\infty}h(u)^2\,\dd u\\
    &\le
    C+C\int_1^{z}u^{2\beta-1}\,\dd u+    Cz
    \int_{z}^{\infty}u^{2\beta-2}\,\dd u\\
    &\le
    C\left(
        1+z^{2\beta}
    \right)+ \frac{C}{1-2\beta}
    z^{2\beta},
\end{split}
\end{align}
where the first inequality uses $\beta\in (0,\frac12)$. When $M\le0$, we set
\[
    r:=-z=\left|z\right|.
\]
Then,  \eqref{K_rep1}  produces the representation
\begin{align}\label{K_rep111}
    \frac{K(M,D)}D
    =
    r\int_{-\infty}^{-r}h(u)^2\,\dd u
    -
    \int_{-r}^0u h(u)^2\,\dd u,\quad M\le 0,\quad D>0.
\end{align}
When $0\le r\le1$,  $h\in L^2(\R)$ and \eqref{K_rep111} ensure $ \frac{K(M,D)}D$ is uniformly bounded. For $r\ge1$, we use the change of variables $v=-u$ in \eqref{K_rep111} and \eqref{hsquared} to produce the bound
\begin{align}\label{new_K_bound2}
\begin{split}
    \frac{K(M,D)}D    &\le
    Cr\int_r^\infty v^{2\beta-2}\,\dd v + \int_{-r}^0|u|h(u)^2\,\dd u\\
    &\le
    Cr^{2\beta}+C+C\int_1^r v^{2\beta-1}\,\dd v\\
    &\le Cr^{2\beta}+C(1+r^{2\beta}).
\end{split}
\end{align}
For all $|M|\le1$ and $0<D\le1$, the bounds \eqref{new_K_bound} and \eqref{new_K_bound2} produce
\begin{align}\label{eq:counter-K-bound}
    K(M,D)
    &\le
    CD\left(        1+|z|^{2\beta}    \right)=
    C\left(        D+D^{1-2\beta}|M|^{2\beta}    \right) 
    \le 2C D^{1-2\beta},
\end{align}
where the last inequality uses $0<\beta<\frac12$ so that $D \le D^{1-2\beta}$.

To apply It\^o's lemma, we need derivatives of the first term in
\eqref{eq:counter-K} given by
\[
    J(M,D):=D\phi\!\left(\frac MD\right),\;M\in \R,\; D>0,\;
    \phi(z):=
    -\int_0^z(z-u)h(u)^2\,\dd u,\; z\in \R.
\]
Since $h$ is continuous,
\[
    \phi'(z)
    =
    -\int_0^z h(u)^2\,\dd u,
    \qquad
    \phi''(z)=-h(z)^2,\quad z\in \R.
\]
Therefore, $\phi\in C^2(\R)$ and so  $J\in C^2$ on
$\R\times(0,\infty)$. Direct differentiation gives
\begin{align*}
    J_M(M,D)
&    =
    -\int_0^{\frac MD}h(u)^2\,\dd u,\quad     J_D(M,D)
    =
    \int_0^{\frac MD}u h(u)^2\,\dd u,\\
    J_{MM}(M,D)    &=
    -\frac1D h(\frac MD)^2,\quad     J_{MD}(M,D)
    =
    \frac{M}{D^2}h(\frac MD)^2,\quad     J_{DD}(M,D)
    =
    -\frac{M^2}{D^3}h(\frac MD)^2.
\end{align*}
For $t<\tau\wedge\rho$, the dynamics in \eqref{dMdD} produce the drift of $J(M_t,D_t)$ to be
\[
\begin{aligned}
&\frac12J_{MM}(M_t,D_t)\Sigma(M_t,D_t)^2
+J_{MD}(M_t,D_t)
  \Sigma(M_t,D_t)\Delta(M_t,D_t)\\
&\qquad
+\frac12J_{DD}(M_t,D_t)\Delta(M_t,D_t)^2\\
&=
-\frac{h(\frac {M_t}{D_t})^2}{2D_t}
\left(
    \Sigma(M_t,D_t)
    -\frac{M_t}{D_t}\Delta(M_t,D_t)
\right)^2\\
&=
-\frac{h(\frac {M_t}{D_t})^2}{2D_t}
\Gamma(M_t,D_t)^2.
\end{aligned}
\]
For the $M^+$ and $M^-$ terms in 
\eqref{eq:counter-K}, Tanaka's formula produces the dynamics
\begin{equation}\label{eq:counter-K-Ito}
\begin{aligned}
    \dd K(M_t,D_t)
    ={}&\dd N_t^K
       -\frac{h(\frac {M_t}{D_t})^2}{2D_t}\Gamma(M_t,D_t)^2\,\dd t
   +\frac{\|h\|_{L^2(\R)}^2}{2}\,\dd L_t^0(M),\quad t< \tau\wedge\rho,
\end{aligned}
\end{equation}
where $N^K$ is a continuous local martingale (its specific form is irrelevant to us) and $L^0(M)$ is $M$'s local time at zero.\footnote{Because $M$ is a continuous local martingale, its left, right, and symmetrical local times are all identical.}  Inserting $\Delta$ from \eqref{eq:counter-profiles} into the dynamics in \eqref{dMdD} produces
$$
\dd D_t=D_t^\beta h(\frac {M_t}{D_t})\,\dd B_t,\quad  
t<\tau\wedge\rho.
$$
Based on these dynamics, It\^o's formula gives
\begin{equation}\label{eq:counter-D-power}
    \dd D_t^{1-2\beta}
    =\dd N_t^D
      -\beta(1-2\beta)
       \frac{h(\frac {M_t}{D_t})^2}{D_t}\,\dd t,\quad  
t<\tau\wedge\rho,
\end{equation}
where $N^D$ is another continuous local martingale (its specific form is irrelevant for us).

By \eqref{eq:counter-Gamma-bound} and \eqref{eq:counter-K-bound}, we can
choose $\varepsilon>0$ so small that,  
\[
\forall |M|\le1\forall D\in (0,1]:   \quad  \varepsilon K(M,D)\le\frac12D^{1-2\beta}
    \quad \text{and}\quad
    \varepsilon\Gamma(M,D)^2\le\beta(1-2\beta).
\]
We define the Lyapunov function
\[
    V(M,D):=D^{1-2\beta}-\varepsilon K(M,D),\quad |M|\le1, \quad 0<D\le1.
\]
The dynamics \eqref{eq:counter-K-Ito} and \eqref{eq:counter-D-power} produce a continuous local martingale $N$ (its specific form is irrelevant for us) such that
\begin{align*}
    \dd V(M_t,D_t) &=\dd N_t+\Big(\frac\varepsilon2\Gamma(M_t,D_t)^2-\beta(1-2\beta)\Big)
       \frac{h(\frac {M_t}{D_t})^2}{D_t}\,\dd t -\frac{\varepsilon\|h\|_{L^2(\R)}^2}{2}\,\dd L_t^0(M)\\
    &\le\dd N_t
       -\frac{\beta(1-2\beta)}2
        \frac{h(\frac {M_t}{D_t})^2}{D_t}\,\dd t-\frac{\varepsilon\|h\|_{L^2(\R)}^2}{2}
         \,\dd L_t^0(M) \\
             &\le\dd N_t
      -\frac{\varepsilon\|h\|_{L^2(\R)}^2}{2}
         \,\dd L_t^0(M),\quad
t<\tau\wedge\rho.
\end{align*}
To control the local time term, we define the stopping time
\begin{align}\label{stopping_lambda}
    \lambda:=\inf\{t\ge0:L_t^0(M)\ge1\}.
\end{align}
For $n>\frac1d$, $n\in \N$, we define the stopping times
\[
    \nu_n:=\inf\{t\ge0:D_t\le\frac1n\},
    \qquad
    S_n:=\lambda\wedge\rho\wedge n\wedge\nu_n,\quad n\in\N.
\]
Since \(n>\frac1d\), we have \(D_0=d>\frac1n\). By $D$'s continuity, all $D$'s 
paths that reach \(0\) must first reach \(\frac1n\) and so
$$
 \nu_n\le\tau,\quad     S_n\le\tau\wedge\rho\wedge n,\quad \text{a.s.}
$$
On the stochastic interval \(t\in [0,S_n]\), we have
\[
    |M_t|\le1,
    \qquad
    \frac1n\le D_t\le1,
    \qquad
    t\le n.
\]
Because all coefficients in the local-martingale part $N$ of $V$ are
bounded on \([0,S_n]\), the stopped local martingale is a square
integrable martingale with mean zero. Because
$M_0=0$, $D_0=d$, and $K(0,d)=0$ (from $K$'s representation in \ref{K_rep1}), taking expectations gives
\[
\E\!\left[V(M_{S_n},D_{S_n})\right]
\le
 d^{1-2\beta}-\frac{\varepsilon\|h\|_{L^2(\R)}^2}{2}
    \E L_{S_n}^0(M).
\]
Because $V\ge \frac12D^{1-2\beta}\ge0$, the left-hand side is nonnegative and so discarding it gives 
\begin{align}\label{localtime1}
    \frac{\varepsilon\|h\|_{L^2(\R)}^2}{2}
    \E L_{S_n}^0(M)
    \le d^{1-2\beta}.
\end{align}
The continuity of $D$ gives $\lim_n \nu_n= \tau$ a.s., and so since $L^0(M)$ is continuous and nondecreasing, we have
\[
   L_{S_n}^0(M)
    \uparrow
    L_{\lambda\wedge\rho\wedge\tau}^0(M),\quad \text{almost surely as}\quad n\to\infty.
\]
Applying the Monotone Convergence Theorem in \eqref{localtime1} produces
\begin{align}\label{localtime2}
    \frac{\varepsilon\|h\|_{L^2(\R)}^2}{2}
    \E L_{\lambda\wedge\rho\wedge\tau}^0(M)
    \le d^{1-2\beta}.
\end{align}
By continuity of $L^0(M)$, on $\{\lambda<\infty\}$ we have
$L_\lambda^0(M)=1$. Hence, for $\omega \in \{\lambda<\tau\wedge\rho\}$, we have
\[
    \mathbf 1_{\{\lambda<\tau\wedge\rho\}}(\omega)
    =1 = L_\lambda^0(M)(\omega)= 
    L_{\lambda\wedge\rho\wedge\tau}^0(M)(\omega).
\]
Consequently, the inequality in \eqref{localtime2} gives the bound
\begin{equation}\label{eq:counter-lambda-bound}
    \P(\lambda<\tau\wedge\rho)
    \le
    \frac{2d^{1-2\beta}}
         {\varepsilon\|h\|_{L^2(\R)}^2}.
\end{equation}

\noindent{\bf Step 4/5:} The process $M$ has dynamics in \eqref{dMdD} with diffusion coefficient
$\Sigma(M_t,D_t)$. The definition \eqref{eq:Sigma-def} ensures $\Sigma \ge 1$ for all $M\in \R$ and $D\ge0$ because  both $\sigma_+,\sigma_{\mathrm{sym}}\ge 1$. Hence, $\langle M\rangle_t\ge t$ and Dambis--Dubins--Schwarz gives a Brownian motion $W$ such that $M_t=W_{\langle M\rangle_t}$. We start by proving
\begin{align}\label{Localtimeprop1} 
    L_t^0(M)=L_{\langle M\rangle_t}^0(W),\quad t\ge0.
\end{align}
The stochastic time-change formula gives
\[
    \int_0^t\operatorname{sgn}(M_s)\,\dd M_s
    =
    \int_0^{\langle M\rangle_t}\operatorname{sgn}(W_u)\,\dd W_u,\quad t\ge0.
\]
By comparing the two Tanaka formulas
\[
    |M_t|
    =
    \int_0^t\operatorname{sgn}(M_s)\,\dd M_s
    +
    L_t^0(M),\quad 
    |W_{\langle M\rangle_t}|
    =
    \int_0^{\langle M\rangle_t}\operatorname{sgn}(W_u)\,\dd W_u
    +
    L_{\langle M\rangle_t}^0(W),
\]
and using \(M_t=W_{{\langle M\rangle_t}}\), we obtain \eqref{Localtimeprop1}.

Next, we define the stopping times
\[
    \eta:=\inf\{t\ge0:|M_t|\ge1\},
    \qquad
    \tau_W:=\inf\{u\ge0:|W_u|=1\}.
\]
While $\P(\tau_W<\infty) =1$ is well-known, we will prove $\P(\eta <\infty)=1$. Because $\langle M\rangle$ is continuous and strictly increasing with range $[0,\infty)$, its inverse is well-defined and is denoted $\langle M\rangle^{-1}$. 
For $t<\langle M\rangle^{-1}_{\tau_W}$, we  have \(\langle M\rangle_t<\tau_W\), and hence
\[
    |M_t|=|W_{\langle M\rangle_t}|<1,\quad t<\langle M\rangle^{-1}_{\tau_W}, \quad     |M_{\langle M\rangle^{-1}_{\tau_W}}|
    =
    |W_{\tau_W}|
    =
    1.
\]
Consequently, $\eta = \langle M\rangle^{-1}_{\tau_W} $ and so $\P(\tau_W<\infty) =1$ implies $\P(\eta <\infty)=1$.  Because  $\eta = \langle M\rangle^{-1}_{\tau_W} $, we have $\langle M\rangle_\eta=\tau_W$ and consequently \eqref{Localtimeprop1} gives
\begin{align}\label{Localtimeprop1111} 
    L_\eta^0(M)
    =
    L_{\langle M\rangle_\eta}^0(W)
    =
    L_{\tau_W}^0(W).
\end{align}

Next, we prove the set identity
\begin{align}\label{setinclusion4}
    \{\lambda<\eta\}
=
    \bigl\{L_\eta^0(M)\ge1\bigr\},
    \qquad \text{a.s.}
\end{align}
Because  $\P(\eta <\infty)=1$, it suffices to consider $\omega \in (\eta <\infty)$ arbitrary when proving  \eqref{setinclusion4}. We first assume \(\lambda(\omega)<\eta(\omega)\). Continuity and non-decreasingness of local time give
\[
    L_{\eta(\omega)}^0(M)(\omega)\ge    L_{\lambda(\omega)}^0(M)(\omega)=1.
\]
Second, we assume $L_{\eta(\omega)}^0(M)(\omega)\ge1$. The continuity of  \(M\) gives \(|M_{\eta}(\omega)|=1\) and so the existence of \(\delta(\omega)>0\) such that \(M_t(\omega)\neq0\) for all 
\(t\in[\eta(\omega)-\delta(\omega),\eta(\omega)]\) follows. Because \(L^0(M)\) is continuous and  \(\dd L_t^0(M)\) is supported on $\{t\ge0:M_t=0\}$, the local time is
constant on $[\eta(\omega)-\delta(\omega),\eta(\omega)]$, that is,
\[
    L_{\eta(\omega)-\delta(\omega)}^0(M)(\omega)
    =
    L_{\eta(\omega)}^0(M)(\omega)
    \ge1.
\]
Because $\lambda$ from \eqref{stopping_lambda} is $L^0(M)$'s first time hitting 1, we get
\[
    \lambda(\omega)
    \le\eta(\omega)-\delta(\omega)
    <\eta(\omega).
\]

The proof of Proposition~27.2 in Bass (2011) shows  $L_{\tau_W}^0(W)\sim \text{Exp(1)}$. 
For completeness, we reprove  $L_{\tau_W}^0(W)\sim \text{Exp(1)}$ in Appendix \ref{Bass}. 
Because \(L_{\tau_W}^0(W)\) has an exponential distribution, it  has no atom at \(1\), and we get
\begin{equation}\label{eq:counter-local-time-probability}
    \P(\lambda<\eta)
    = \P  \big(L_\eta^0(M)\ge1\big) = 
    \P\!\left(L_{\tau_W}^0(W)\ge1\right)
    =    \P\!\left(L_{\tau_W}^0(W)>1\right) = 
    e^{-1},
\end{equation}
where the first equality follows from \eqref{setinclusion4} and the second uses \eqref{Localtimeprop1111}.

Comparison and the dynamics in \eqref{dMdD} ensure that \(D\) is a nonnegative local martingale, hence, $D$ is a
supermartingale.  Ville's maximal inequality therefore gives
\[
    \P\Big(\sup_{t\in[0,\infty)}D_t\ge1\Big)\le  D_0
    =
d.
\]
The inequality $\P(A \cap B^c) \ge \P(A) -\P(B)$  gives
\[
    \P(\mathcal E)
    \ge
    \P\!\left(
        \{\lambda<\eta\}
        \cap
        \left\{\sup_{t\in[0,\infty)}D_t<1\right\}
    \right)
    \ge e^{-1}-d,
    \quad
    \mathcal E
    :=
    \{\lambda<\eta\}
    \cap
    \left\{
        \sup_{t\in[0,\lambda]}D_t<1
    \right\}.
\]

To conclude the proof, we need the following two set inclusions
\begin{align}\label{inclusions1}
 \mathcal E\cap\{\tau\le\lambda\}&\subseteq\{\tau<\rho\},
 &
 \mathcal E\cap\{\tau>\lambda\}
 &\subseteq\{\lambda<\tau\wedge\rho\}.
\end{align}
To see the first set inclusion in \eqref{inclusions1}, we first let $\omega \in \mathcal E$ be arbitrary. Then, the first part of $\sE$ gives $\lambda(\omega)<\eta(\omega)$ and so $\lambda(\omega)<\infty$. Using also the second part of \(\mathcal E\) gives
$$ |M_t(\omega)|<1 \quad\text{and}\quad D_t(\omega)<1, \qquad t\in[0,\lambda(\omega)] \subset [0,\eta(\omega)). 
$$
Therefore, the definition of $\rho$ in \eqref{taurho} gives $ \lambda(\omega)<\rho(\omega). $ When we further have  $\omega \in \mathcal E\cap\{\tau\le\lambda\}$, it also holds that $ \tau(\omega)\le\lambda(\omega)<\rho(\omega)$ and the first inclusion in \eqref{inclusions1} holds. To see the second set inclusion in \eqref{inclusions1}, we let $\omega \in   \mathcal E\cap\{\tau>\lambda\}$ be arbitrary. Because $\omega \in \sE$, we have
$$
 |M_t(\omega)|<1,\qquad D_t(\omega) <1,\qquad t \in [0,\lambda(\omega)],
$$
Therefore, the definition of $\rho$ in \eqref{taurho} gives  $\lambda(\omega)<\rho(\omega)$. Trivially,  $\omega \in   \mathcal E\cap\{\tau>\lambda\}$ gives $\lambda(\omega)<\tau(\omega)$ and so the minimum $\tau(\omega)\land \rho(\omega)$ is strictly bigger than $\lambda(\omega)$ and the second inclusion in \eqref{inclusions1} follows.

The equality $\P(A \cap B^c) = \P(A) -\P(A \cap B)$, the set inclusions \eqref{inclusions1}, and the bound in \eqref{eq:counter-lambda-bound}    give
\[
\begin{aligned}
    \mathbb P(\tau<\rho)
    &\ge
    \mathbb P\bigl(\mathcal E\cap(\tau\le\lambda)\bigr)\\
    &=
    \mathbb P(\mathcal E)     -  \P\bigl(\mathcal E\cap(\tau>\lambda)\bigr)\\
    &\ge 
        \mathbb P(\mathcal E)-     \mathbb P(\lambda<\tau\wedge\rho)\\
        &\ge e^{-1}-d-
    \frac{2d^{1-2\beta}}
         {\varepsilon\|h\|_{L^2(\mathbb R)}^2}.
\end{aligned}
\]
This lower bound is positive for all sufficiently small $d>0$ ( $\varepsilon$ does not depend on $d$).

\noindent{\bf Step 5/5:}
On $\{\tau<\rho\}$, for every $t\le\tau$,
\[
 |X_t^{\pm \frac d2}|
 \le |M_t|+\frac12D_t<\frac32.
\]
Hence, up to $\tau$, the paths remain in the region where the
truncated and untruncated coefficients coincide. By pathwise
uniqueness, the  solutions corresponding to the truncated and untruncated coefficients coincide up to $\tau$.
Therefore, the untruncated equations also have a finite-time collision
with positive probability.
\end{proof}

\subsection{Consequences for subquadratic Sobolev-flow results}

First, we discuss  the exponent issue in Zhang (2011).
When specialized to the
zero-drift time-homogeneous SDE \eqref{eq:intro-sde} in one dimension, the Sobolev integrability
condition  in Theorem 1.1  in Zhang (2011) becomes
\[
    \sigma'\in
   L^p(\mathbb R), 
    \qquad
    p\in(1,\infty).
    \]

Fix $p\in(1,2)$ arbitrary and choose $\beta$ such that $1-\frac1p<\beta<\frac12$. The coefficient $\sigma_\text{sym}$ in \eqref{both_trunc} is bounded, uniformly continuous, uniformly
bounded away from zero, and $\sigma_\text{sym}'\in L^p(\R)$. This implies that  $\sigma_\text{sym}$ satisfies the regularity assumptions of Theorem~1.1 in 
Zhang (2011). However, this conflicts with our Theorem~\ref{thm:main}, which ensures that  there are $x<y$ such that the stopping time
\[
    \tau:=\inf\{t\ge0:X_t^x=X_t^y\},
\]
satisfies $\P(\tau<\infty)>0$.  Hence, for some deterministic
$N\in\N$, it must be $ \P(\tau\le N)>0$. Pathwise uniqueness implies that the two solutions remain equal after
their first meeting, and consequently
\[
    \P(X_N^x=X_N^y)>0.
\]
Therefore,  the map
\[
  \R\ni  z\longmapsto X_N^z
\]
fails to be injective with positive probability.  This contradicts the conclusion of Theorem~1.1 in Zhang (2011) that, for each $t\ge0$, $z\mapsto X_t^z$ is almost surely a homeomorphism.

Second, we discuss Theorem~5.1(1) of Ren and Zhang (2024). Fix
\(p\in(1,2)\) and choose $\beta \in (1-\frac1p,\frac12)$  such that \eqref{W1pmembership} holds. 
 We use the truncated coefficient  \(\sigma:=\sigma_{\mathrm{sym}}\) in
\eqref{both_trunc} to define the Lamperti map
\[
   F(x):=\int_0^x\frac{\dd u}{\sigma(u)}, \quad x\in \R.
\]
Since \(\sigma\) is bounded above and bounded away from zero,
\(F\) is a strictly increasing global Lipschitz \(C^1\)-bijection with also $F^{-1}$ globally Lipschitz.
Moreover we have
\[
   F''(x)=-\frac{\sigma'(x)}{\sigma(x)^2},\quad x\in \R\setminus\{-2,0,2\}.
\]
Because \(F'(x)\sigma(x)=1\) and
\(F''(x)\sigma(x)^2=-\sigma'(x)\) for $x\in \R\setminus\{-2,0,2\}$, the generalized It\^o formula applied to
\(Y_t:=F(X_t)\) gives the dynamics in \eqref{Lamperti_Dynamics}. To see  \(b \in L^p(\mathbb R)\), 
the change of variables \(y:=F(x)\) gives
\begin{equation}\label{eq:Lamperti-Lp}
 \int_{\mathbb R}|b(y)|^p\,\dd y
 =
 2^{-p}\int_{\mathbb R}
       \frac{|\sigma'(x)|^p}{\sigma(x)}\,\dd x
 <\infty ,
\end{equation}
where finiteness follows from \(\sigma'(x)=0\) for $x\notin [-2,2]$ whereas \(p(\beta-1)>-1\) gives 
integrability of $ |\sigma'(x)|^p=\beta^p |x|^{p(\beta-1)}$ for $x \in (-2,0)\cup (0,2)$.
 By choosing $q>1$ sufficiently large, we see that $b$  in \eqref{Lamperti_Dynamics} satisfies the LPS condition  \((H^b_{p,q})\) of Ren and Zhang (2024). 

For the coefficients in \eqref{both_trunc}, the Lamperti transformation is reversible. To this end, we define 
$G:= F^{-1}$. It\^o's generalized formula gives that $G(Y^{F(z)}_t)$ satisfies \eqref{eq:intro-sde} and so pathwise uniqueness gives  $X^z_t = G(Y^{F(z)}_t)$ for $z\in \R$.

The unit diffusion coefficient in \eqref{Lamperti_Dynamics}
satisfies both the uniform ellipticity/uniform-continuity assumption
\((H^\sigma_e)\) and the diffusion assumption
\((\widehat H^\sigma_H)\) in Theorem~5.1(1) of Ren and Zhang (2024).
Because \(F\) is injective, we have
\[
 X_t^x=X_t^y
 \quad\Longleftrightarrow\quad
 F(X_t^x)=F(X_t^y)
 \quad\Longleftrightarrow\quad
 Y_t^{F(x)}=Y_t^{F(y)}.
\]
However, our Theorem~\ref{thm:main} contradicts the strict
non-confluence conclusion of Theorem~5.1(1) of Ren and Zhang (2024).

\section{Insufficiency of $C^{\beta}(\R)$ for $\beta \in [\frac12,1)$} \label{sec:critical-holder-counterexample}

For all $\beta\in(0,\frac12)$, the corresponding coefficients
in \eqref{both_trunc} belong to $C^\beta(\mathbb R)$ and fail
strict comparison. We now construct counterexamples with
$C^\beta(\mathbb R)$ regularity for all
$\beta\in[\frac12,1)$. Fix any $\alpha\in(\frac12,1)$ and define the $2\pi$-periodic function (Weierstrass function)
\begin{equation}\label{eq:little-o-f}
    f_\alpha(x):=\sum_{n=0}^\infty 2^{-\alpha n}\cos(2^n x),
    \qquad x\in\mathbb R.
\end{equation}
This function is  nowhere differentiable by Hardy's theorem (because $2^{1-\alpha}>1$). 
Since any function in $W^{1,1}_{\mathrm{loc}}(\mathbb R)$
is locally absolutely continuous (see, e.g., Theorem 7.13 in Leoni, 2009), we see $f_\alpha \notin W^{1,1}_{\mathrm{loc}}(\mathbb R)$.

For a constant $\varepsilon>0$ to be chosen below, we define the coeffient
\begin{equation}\label{eq:critical-sigma}
    \sigma(x):=1+\varepsilon f_\alpha(x),
    \qquad x\in\R.
\end{equation}

\begin{theorem}\label{thm:critical-holder-counterexample} Fix  $\alpha\in(\frac12,1)$.  
There exists $\varepsilon_0>0$ such that, for all
$0<\varepsilon\le\varepsilon_0$, the coefficient
\eqref{eq:critical-sigma} is  bounded, bounded away from
zero, and belongs to $C^\alpha(\R)\subset C^{\frac12}(\R)$. 
For all initial points $z\in \R$, the SDE \eqref{eq:intro-sde} has a global
strong solution and pathwise uniqueness holds. Furthermore, there exist $x<y$ such that
\[
    \mathbb P\!\left(
       X_t^x=X_t^y
       \text{ for some }t<\infty
    \right)>0.
\]
\end{theorem}

\begin{proof}
We divide the proof into several steps.

\medskip
\noindent{\bf Step 1/7:} This step uses the properties
\begin{align}\label{sum_prop}
    S_\alpha:=\sum_{n=0}^\infty 2^{-\alpha n}=\frac{2^{\alpha }}{2^{\alpha }-1}<\infty,\quad    \sum_{n=0}^\infty 2^{(1-\alpha) n}=\infty.
\end{align}
Finiteness of the first sum in \eqref{sum_prop} and Weierstrass \(M\)-test ensure the series in \eqref{eq:little-o-f} converges uniformly, and so  $f_\alpha$
is a bounded, continuous function with period $2\pi$, and $\|f_\alpha\|_\infty\le S_\alpha$. 
Furthermore, the series in \eqref{sum_prop} imply that  \(2^{(1-\alpha)n}\) can only be summed over a finite range whereas $2^{-\alpha n}$ is summable. We have the geometric bounds
\begin{align}\label{two_series}
\sum_{n=0}^N 2^{(1-\alpha)n} \le C_\alpha 2^{(1-\alpha)N},\quad \sum_{n=N+1}^\infty 2^{-\alpha n} \le C_\alpha 2^{-N\alpha},\quad N\in\N_0,
\end{align}
where $C_\alpha >0$ is an irrelevant constant that only depends on $\alpha$ and will change throughout the proof.

We define
\begin{align}\label{def_a_A}
a_n(D):=2^{1-\alpha n}\sin(2^{n-1}D),\quad   A_n(D):=\sum_{ j=0}^{n-1}|a_j(D)|,\quad  n\in\N_0,\quad 0<D\le1.
\end{align}
The sequence $(a_n(D))_{n\in\N}$ is absolutely summable with $|| a(D)||_{\ell_1}\le 2S_\alpha$ where $S_\alpha$ is the first sum in \eqref{sum_prop}.  For $D\in(0,1]$, we choose $N=N(D)\ge0$ so that
$1\le2^ND<2$, which gives
\begin{align}\label{choosing_D}
    \frac D2<2^{-N}\le D,
    \qquad
    \frac12\le2^{N-1}D<1.
\end{align}
By combining \eqref{two_series} and \eqref{choosing_D} with $|\sin(u)|\le |u|$ for $n\le N$ and 
  $|\sin(u)|\le 1$ for $n>N$ we get
\begin{align}\label{manybounds}
|a_n(D)|\le
\begin{cases}
D\,2^{(1-\alpha)n}, & n\le N,\\[2mm]
2^{1-\alpha n}, & n>N,
\end{cases}
\qquad
A_n(D)\le
\begin{cases}
C_\alpha D\,2^{(1-\alpha)n}, & n\le N,\\[2mm]
C_\alpha D^\alpha, & n>N.
\end{cases}
\end{align}

To see $f_\alpha\in C^\alpha(\mathbb R)$, we first consider $D:=|x-y|<1$.  For \(n\le N\), the Mean-Value Theorem 
gives   $  |\cos(2^n x)-\cos(2^n y)|     \le 2^nD$ whereas for \(n>N\) we use $    |\cos(2^n x)-\cos(2^n y)|\le2$. The bounds in \eqref{two_series} produce
\[
\begin{aligned}
 |f_\alpha(x)-f_\alpha(y)|
 &\le D\sum_{n=0}^N 2^{(1-\alpha)n}
      +2\sum_{n>N}2^{-\alpha n}  \\
 &\le D C_\alpha \,2^{(1-\alpha)N}
      +C_\alpha 2^{-\alpha N} \\
      &\le C_\alpha D^\alpha,
\end{aligned}
\]
For $|x-y|\ge 1$, we trivially have
\begin{align}\label{boundedness_holder}
 |f_\alpha(x)-f_\alpha(y)|\le 2||f_\alpha||_\infty \le 2||f_\alpha||_\infty |x-y|^\alpha.
\end{align}
Because $\alpha>\frac12$, the global $C^\alpha$ estimate together with boundedness of $\sigma$ implies  --- similarly to \eqref{boundedness_holder} ---  that 
$\sigma\in C^{\frac12}(\mathbb R)$. Set $\varepsilon_0:= \frac1{2S_\alpha}$. For $\varepsilon\in (0,\varepsilon_0]$, $\sigma $ is uniformly positive too  (use \(\|f_\alpha\|_\infty\le S_\alpha\) so that  
 $  \sigma(x)\ge1-\varepsilon S_\alpha$). The value of \(\varepsilon_0\) will be decreased further below if necessary.

Any $C^{\frac12}(\R)$ coefficient satisfies \eqref{NakaoLeGall}, and so
pathwise uniqueness holds for \eqref{eq:intro-sde}. Together with strong
existence and time homogeneity, this gives the strong Markov property.
Furthermore, as already explained at the beginning of the proof of
Theorem~\ref{thm:general-criterion}, the strong solution is global
(i.e., $\xi^z=\infty$ for all $z\in\R$). Because strong existence and
pathwise uniqueness also hold for the time-homogeneous SDE
\eqref{dMdD}, the strong Markov property holds for $(M,D)$ as well.

\noindent\textbf{Step 2/7:} Let $\Delta$ be defined in \eqref{eq:Delta-def} and define
\begin{align}\label{qD_barqD}
    q_D(M):=\frac{\Delta(M,D)^2}{D},
    \quad 
\overline q_D:=\frac1{2\pi}\int_0^{2\pi}q_D(x)\,\dd x,\quad M\in\mathbb R,\quad D\in (0,1],
\end{align}
 and set
\begin{align}\label{beta_in_terms_alpha}
    \beta:=2\alpha-1\in(0,1).
\end{align}
In this step, we prove existence of constants $0<c_\alpha<C_\alpha<\infty$,
independent of $D$ and $\varepsilon$, such that
\begin{equation}\label{eq:little-o-energy}
    c_\alpha\varepsilon^2D^\beta
    \le \overline q_D
    \le C_\alpha\varepsilon^2D^\beta,
    \qquad 0<D\le1,
\end{equation}
and
\begin{equation}\label{eq:little-o-q-sup}
    0\le q_D(M)\le C_\alpha\varepsilon^2D^\beta,
    \qquad M\in\mathbb R,\quad 0<D\le1.
\end{equation}
The definitions of $\Delta$ in  \eqref{eq:Delta-def}, $f_\alpha$ in \eqref{eq:little-o-f}, and the identity
$\cos(a+b)-\cos(a-b)=-2\sin(a)\sin(b)$ give
\begin{equation}\label{eq:little-o-Delta-fourier}
    \Delta(M,D)
    =-2\varepsilon\sum_{n=0}^\infty
       2^{-\alpha n}\sin(2^nM)\sin(2^{n-1}D),\quad M\in\R,\quad 0<D\le1.
\end{equation}
Finiteness of the first sum in \eqref{sum_prop} ensures the series in \eqref{eq:little-o-Delta-fourier} converges absolutely and uniformly
in \((M,D)\in\mathbb R^2\). Because $q_D$ involves squaring $\Delta$,  we use \eqref{def_a_A} to write
\begin{align}\label{def_an}
\Delta(M,D)=-\varepsilon\sum_{n=0}^\infty a_n(D)\sin(2^nM), \quad M\in\mathbb R, \quad 0<D\le 1. 
\end{align}
Because  $|| a(D)||_{\ell_1}<\infty$, we also have
$a(D)\in\ell_2$ and we get

\begin{align}\label{eq:little-o-energy-formula} 
\begin{split}
\overline q_D
&=
\frac{1}{2\pi D}
\int_0^{2\pi}\Delta(M,D)^2\,\dd M\\
&=
\frac{\varepsilon^2}D
\sum_{n,j=0}^\infty
a_n(D)a_j(D)
\frac1{2\pi}
\int_0^{2\pi}
\sin(2^nM)\sin(2^jM)\,\dd M\\
&=\frac{\varepsilon^2}{2D}\sum_{n=0}^\infty a_n(D)^2\\
&= 
\frac{2\varepsilon^2}D
\sum_{n=0}^\infty
2^{-2\alpha n}\sin(2^{n-1}D)^2,\quad 0<D\le 1,
\end{split}
\end{align}
where the second last equality uses the orthogonality
\[
\frac1{2\pi}
\int_0^{2\pi}
\sin(2^nM)\sin(2^jM)\,\dd M
=
\frac12\mathbf 1_{\{n=j\}}.
\]

We start by proving \eqref{eq:little-o-energy}. The inequality $2^{-N}>\frac D2$ from \eqref{choosing_D} gives
\[
2^{-2\alpha N}
=
(2^{-N})^{2\alpha}
\ge
\left(\frac D2\right)^{2\alpha}
=
2^{-2\alpha}D^{2\alpha}.
\]
From  \eqref{choosing_D} we have $2^{N-1}D \in [\frac12,1]$. Since \(\sin\) is increasing on \([0,1]\), we have
\[
\sin(2^{N-1}D)^2
\ge
\sin\!\left(\frac12\right)^2.
\]
Keeping only the $n=N$ term in
\eqref{eq:little-o-energy-formula} gives the lower bound  in \eqref{eq:little-o-energy}

\[
\frac{\overline q_D}{\varepsilon^2}
\ge
\frac2D\,2^{-2\alpha N}
\sin(2^{N-1}D)^2\ge
2^{1-2\alpha}
\sin\!\left(\frac12\right)^2
D^{2\alpha-1}.
\]

The formula in \eqref{eq:little-o-energy-formula}  and the bounds from \eqref{manybounds} give the upper bound in \eqref{eq:little-o-energy}:
\[
\begin{aligned}
\overline q_D
&=
\frac{\varepsilon^2}{2D}
\sum_{n=0}^\infty a_n(D)^2\\
&\le
\frac{\varepsilon^2}{2D}
\left(
D^2\sum_{n=0}^N2^{2(1-\alpha)n}
+
4\sum_{n=N+1}^\infty2^{-2\alpha n}
\right)\\
&\le
C_\alpha\varepsilon^2
\left(
D\,2^{2(1-\alpha)N}
+
D^{-1}2^{-2\alpha N}
\right)\\
&\le
C_\alpha\varepsilon^2D^{2\alpha-1},\quad 0<D\le 1,
\end{aligned}
\]
where the second last inequality uses
$$ 
\sum_{n=0}^N2^{2(1-\alpha)n} = \frac{2^{2(1-\alpha)(N+1)}-1} {2^{2(1-\alpha)}-1} \le C_\alpha 2^{2N(1-\alpha)}
 $$
  and the last inequality uses \eqref{choosing_D}. The upper bound in \eqref{eq:little-o-q-sup} follows similarly from  \eqref{def_an} and the bounds from \eqref{manybounds}:
\[
\begin{aligned}
|\Delta(M,D)|
&\le
\varepsilon\sum_{n=0}^\infty|a_n(D)|\\
&\le
\varepsilon
\left(
D\sum_{n=0}^N2^{(1-\alpha)n}
+2\sum_{n=N+1}^\infty2^{-\alpha n}
\right)\\
&\le
C_\alpha\varepsilon
\left(
D2^{(1-\alpha)N}
+
2^{-\alpha N}
\right)\\
&\le
C_\alpha\varepsilon D^\alpha,
\end{aligned}
\]
where the second last inequality uses \eqref{two_series} and the last inequality uses \eqref{choosing_D}. Consequently, the upper bound in  \eqref{eq:little-o-q-sup} follows
\[
q_D(M)
=
\frac{\Delta(M,D)^2}{D}
\le
C_\alpha\varepsilon^2D^{2\alpha-1}.
\]

\noindent\textbf{Step 3/7:} Fix $D\in(0,1]$ in this step. We define the kernel
\[
    G(t):=\frac{|t|}{2}-\frac{t^2}{4\pi}-\frac{\pi}{6}.
    \qquad t\in[-\pi,\pi],
\]
Because $G(\pi)=G(-\pi)$, we can extend $G$ to \(\mathbb R\) using \(2\pi\) periods. Then, the function
\begin{align}\label{green_poisson}
\psi_D( M):=
    2\int_0^{2\pi}
    G(M-s)\bigl(q_D(s)-\overline q_D\bigr)\,\dd s,\quad M\in \R,
\end{align}
is the unique zero mean $2\pi$-periodic classical solution in $C^2(\R)$ of the Poisson equation\footnote{More generally, distributionally uniqueness among all \(2\pi\)-periodic, mean-zero continuous functions holds for \eqref{eq:little-o-corrector-eq}.}
\begin{equation}\label{eq:little-o-corrector-eq}
    \psi_D''(M)=2\bigl(q_D(M)-\overline q_D\bigr),\quad M\in \R.
\end{equation}

This step is devoted to proving
\begin{equation}\label{eq:little-o-corrector-bound}
    \|\psi_D\|_\infty\le C_\alpha\varepsilon^2D,\quad D\in (0,1].
\end{equation}
When squaring $\Delta (M,D)$ in \eqref{qD_barqD}, we will use 
$    \sin(a)^2=\frac12-\frac12\cos(2a)$ for the diagonal terms and 
$
    2\sin a\sin b=\cos(a-b)-\cos(a+b)
$
for the off-diagonal terms. Because $a(D)\in \ell_1$, we insert the terms from  \eqref{qD_barqD} to rewrite the right-hand side of \eqref{eq:little-o-corrector-eq} as
\begin{align}\label{star1}
\begin{split}
2\bigl(q_D(M)-\overline q_D\bigr)
=
\frac{\varepsilon^2}{D}\Bigg[
&-\sum_{n=0}^\infty
 a_n(D)^2\cos(2^{n+1}M)\\
&+2\sum_{0\le j<n}
 a_n(D)a_j(D)
 \cos\bigl((2^n-2^j)M\bigr)\\
&-2\sum_{0\le j<n}
 a_n(D)a_j(D)
 \cos\bigl((2^n+2^j)M\bigr)
\Bigg].
\end{split}
\end{align}
Because $a(D)\in\ell_1$, these series are absolutely and uniformly
convergent. By using the integral identity
\begin{equation}\label{green_cosine}
    \int_0^{2\pi}G(M-s)\cos(ks)\,\dd s
    =
    -\frac{\cos(kM)}{k^2},
    \qquad k\in \N,
\end{equation}
when inserting \eqref{star1} into the right-hand side of \eqref{green_poisson}, we see 
\begin{align}\label{psiD2}
\begin{split}
\psi_D(M)
=
\frac{\varepsilon^2}{D}\Bigg[
&\sum_{n=0}^\infty
 \frac{a_n(D)^2}{(2^{n+1})^2}\cos(2^{n+1}M)\\
&-2\sum_{0\le j<n}
 \frac{a_n(D)a_j(D)}
 {(2^n-2^j)^2}
 \cos\bigl((2^n-2^j)M\bigr)\\
&+2\sum_{0\le j<n}
 \frac{a_n(D)a_j(D)}
 {(2^n+2^j)^2}
 \cos\bigl((2^n+2^j)M\bigr)
\Bigg].
\end{split}
\end{align}
To see that \eqref{psiD2} satisfies the bound \eqref{eq:little-o-corrector-bound}, we note for  \(0\le j<n\), we trivially have $2^n-2^j\ge2^{n-1}$. Therefore, 
\[
    \frac1{(2^n-2^j)^2}
    \le4\,2^{-2n},
    \quad
    \frac1{(2^n+2^j)^2}
    \le2^{-2n},\quad     \frac1{(2^{n+1})^2}\le2^{-2n},\quad 0\le j<n.
\]
Consequently, by using $A_n(D)$ from \eqref{def_a_A},  we can bound \eqref{psiD2} as 
\begin{align}\label{green_direct_bound}
    \|\psi_D\|_\infty
    \le
    \frac{C\varepsilon^2}{D}
    \sum_{n=0}^\infty
    2^{-2n}
    \bigl(a_n(D)^2+|a_n(D)|A_n(D)\bigr).
\end{align}
By using $N=N(D)$ from \eqref{choosing_D}, we can split into two sums. First, the bound $\sum _{n=0}^N 2^{-2\alpha n} \le \frac{4^\alpha}{4^\alpha-1}$ and  \eqref{manybounds} give
\begin{align*}
    \frac{C\varepsilon^2}{D}
    \sum_{n=0}^N
    2^{-2n}
    \bigl(a_n(D)^2+|a_n(D)|A_n(D)\bigr)
&\le     C_\alpha\varepsilon^2 D
    \sum_{n=0}^{N}2^{-2n}
\bigl(2^{2(1-\alpha)n}+C_\alpha 2^{2(1-\alpha)n}\bigr)\\
&\le
C_\alpha\varepsilon^2 D.
\end{align*}
Then, because $ \sum _{n=N+1}^{\infty } 2^{-2n (1+\alpha ) }=\frac{2^{-2 (1+\alpha) N}}{2^{2 \alpha +2}-1}$ and 
$ \sum _{n=N+1}^{\infty } 2^{-n (2+\alpha ) } = \frac{2^{-(2+\alpha )N}}{2^{\alpha +2}-1}$,  the bound in  \eqref{green_direct_bound} and  \eqref{manybounds} give
\begin{align*}
    \|\psi_{D}\|_\infty
    &\le C_\alpha \varepsilon^2   \Big(   D+\frac1D
    \sum_{n=N+1}^{\infty}\bigl(2^{2-2n(1+\alpha)}+2^{1-n(2+\alpha)}C_\alpha D^\alpha\bigr)\Big)\\
&\le C_\alpha\varepsilon^2   \Big(  D+\frac1DC_\alpha 2^{-2 (1+\alpha ) N}+2^{-(2+\alpha )N}C_\alpha D^{\alpha-1}\Big)\\
&\le C_\alpha\varepsilon^2   \Big(  D+ D^{1+2\alpha}\Big),
\end{align*}
where the last inequality uses  \eqref{choosing_D}. Because $D\in (0,1]$, we get \eqref{eq:little-o-corrector-bound}.

\medskip
\noindent\textbf{Step 4/7:}
Let $1<r\le2$.  For
\[
    \frac Dr\le \widetilde D\le rD,
    \qquad 0<D\le\frac1r,
\]
we claim
\begin{equation}\label{eq:little-o-stability}
    \sup_{M\in\mathbb R}
    |q_{\widetilde D}(M)-q_D(M)|
    \le C_\alpha\varepsilon^2D^{2\alpha-1}(r-1)^\alpha.
\end{equation}
To see \eqref{eq:little-o-stability},  the definition of $\Delta$ 
in \eqref{eq:Delta-def} gives
\begin{align}\label{alphaest1} 
\begin{split}
\Delta(M,\widetilde D)-\Delta(M,D) ={}&\varepsilon  \Big[ f_\alpha(M+
\widetilde D/2)-f_\alpha(M+D/2)\\ &\qquad\qquad -f_\alpha(M-\widetilde D/2)+f_\alpha(M-D/2) \Big]. 
\end{split}
\end{align}
The $\alpha$-H\"older estimate for $f_\alpha$  in Step 1 gives
$$ \left|f_\alpha(M+\widetilde D/2)-f_\alpha(M+D/2)\right| \le C_\alpha\left(\frac{|\widetilde D-D|}{2}\right)^\alpha, $$
and similarly for the other terms in \eqref{alphaest1}. This gives the uniform (in $M\in \R)$ $\alpha$-H\"older estimate 
\begin{align*}
 |\Delta(M,\widetilde D)-\Delta(M,D)|
 &\le C_\alpha\varepsilon|\widetilde D-D|^\alpha.
 \end{align*}
 Because \(\widetilde D\le rD\le2D\), the $\alpha$-H\"older estimate for $f_\alpha$  in Step 1 also gives
$$ \begin{aligned} |\Delta(M,\widetilde D)|+|\Delta(M,D)| &\le \varepsilon C_\alpha \bigl(\widetilde D^\alpha+D^\alpha\bigr)\le \varepsilon C_\alpha (2^\alpha+1)D^\alpha\le C_\alpha\varepsilon D^\alpha,
\end{aligned} $$
where the last inequality uses an enlarged \(C_\alpha\). These bounds, $a^2-b^2 = (a-b)(a+b)$, and the definition of $q_D$ in \eqref{qD_barqD} produce
\[
\begin{aligned}
 |q_{\widetilde D}(M)-q_D(M)|
 &\le
 \frac{|\Delta(M,\widetilde D)-\Delta(M,D)|
       (|\Delta(M,\widetilde D)|+|\Delta(M,D)|)}{\widetilde D}\\
 &\quad
 +\Delta(M,D)^2
   \left|\frac1{\widetilde D}-\frac1D\right|\\
 &\le C_\alpha\varepsilon^2\Big( \frac{|\widetilde D-D|^\alpha}{\widetilde D}D^\alpha + D^{2\alpha} \frac{|D-\widetilde D|}{\widetilde D D}\Big)\\
  &\le
 C_\alpha\varepsilon^2 \bigl((r-1)^\alpha D^{2\alpha -1}  r+(r-1) D^{2\alpha -1}  r\bigr),
\end{aligned}
\]
where the  last inequality uses $|\widetilde D-D| \le (r-1)D$ and $\frac1{\widetilde D}\le \frac rD$.
Since $r\in(1, 2]$, $0<r-1\le1$, and $\alpha<1$, we have
$r(r-1)\le 2(r-1)^\alpha$, which proves \eqref{eq:little-o-stability}.

Enlarging $C_\alpha$ if necessary, we take the same
$C_\alpha$ so that all the upper bounds in
\eqref{eq:little-o-energy}, \eqref{eq:little-o-q-sup}, and
\eqref{eq:little-o-stability} hold simultaneously. Recall that $c_\alpha:= 2^{1-2\alpha}
\sin\!\left(\frac12\right)^2$ is the constant in the lower bound in \eqref{eq:little-o-energy}.  We choose \(r_0\in(1,2]\) sufficiently close to \(1\) that
\begin{equation}\label{eq:little-o-r-choice}
    C_\alpha(r-1)^\alpha\le\frac{c_\alpha}{4},\quad r\in (1,r_0),
\end{equation}
and then we  fix such \(r\in(1,r_0)\). The definition of $\Sigma$  in \eqref{eq:Sigma-def} and $\|f_\alpha\|_\infty\le S_\alpha$ produce 
\[
    |\Sigma(M,D)-1|\le \frac \varepsilon2\Big| f_\alpha (M+\frac D2) +  f_\alpha (M-\frac D2)\Big| \le \varepsilon S_\alpha,
    \quad M\in\mathbb R,\quad D\ge0.
\]
Because $a^2-1 = 2(a-1) + (a-1)^2$, we get the uniform bound
\[
 \sup_{M\in\mathbb R,D\ge0}|\Sigma(M,D)^2-1|
 \le 2\varepsilon S_\alpha+\varepsilon^2S_\alpha^2.
\]
In addition to $\varepsilon_0$ being such that $\sigma $ is uniformly positive, if needed, we can lower $\varepsilon_0$  such that 
\begin{equation}\label{eq:little-o-eps-choice}
\big(2\varepsilon S_\alpha+\varepsilon^2S_\alpha^2\big) C_\alpha
    \le\frac{c_\alpha}{4},
    \qquad 0<\varepsilon\le\varepsilon_0.
\end{equation}

\noindent\textbf{Step 5/7:} Fix $M\in\mathbb R$, $r\in (1,r_0)$, $0<D\le1/r$, and consider the processes $(M_t,D_t)$  in \eqref{dMdD} with
$M_0=M$ and $D_0=D$, and define the stopping time
\begin{align}\label{theta_D}
    \theta_D
    :=\inf\left\{t\ge0:D_t\notin\left(\frac Dr,rD\right)\right\}.
\end{align}

This step proves
\begin{equation}\label{eq:little-o-exit}
    \mathbb E[\theta_D]
    \le
   C_{\alpha,\varepsilon}
    D^{2-2\alpha},\quad C_{\alpha,\varepsilon}:=  \frac{2r(r-1)^2+4C_\alpha\varepsilon^2}
         {c_\alpha\varepsilon^2}.
\end{equation}
Let $\psi_D$ be as in \eqref{green_poisson}. Because \(\psi_D\in C^2(\mathbb R)\) and is periodic, 
\(\psi_D'\) is uniformly bounded. Since  also \(\Sigma\) in \eqref{eq:Sigma-def} is uniformly bounded, the
stochastic integral below,  is a square
integrable martingale with mean zero on any finite time interval. For a constant $T\in(0,\infty)$, It\^o's lemma applied to \(M\mapsto\psi_D(M)\) gives
\[
\begin{aligned}
\psi_D(M_{\theta_D\wedge T})-\psi_D(M_0)
={}&
\int_0^{\theta_D\wedge T}
\psi_D'(M_s)\Sigma(M_s,D_s)\,\dd B_s\\
&+\frac12\int_0^{\theta_D\wedge T}
\psi_D''(M_s)\Sigma(M_s,D_s)^2\,\dd s.
\end{aligned}
\]
Using the dynamics \eqref{dMdD} and the Poisson ODE \eqref{eq:little-o-corrector-eq}, we obtain
\[
\begin{aligned}
 \mathbb E\int_0^{\theta_D\wedge T}q_D(M_s)\,\dd s
 ={}&\overline q_D\,\mathbb E[\theta_D\wedge T]
 +\mathbb E[\psi_D(M_{\theta_D\wedge T})-\psi_D(M_0)]\\
 &-\mathbb E\int_0^{\theta_D\wedge T}
   (\Sigma(M_s,D_s)^2-1)
   (q_D(M_s)-\overline q_D)\,\dd s.
\end{aligned}
\]
We will bound each term on the right-hand side. From \eqref{eq:little-o-energy}, we have
$$ 
\bar q_D\,\mathbb E[\theta_D\wedge T]\ge c_\alpha\varepsilon^2D^\beta\mathbb E[\theta_D\wedge T]. 
$$
From \eqref{eq:little-o-corrector-bound}, we have
$$ 
\mathbb E[\psi_D(M_{\theta_D\wedge T})-\psi_D(M_0)] \ge-2\|\psi_D\|_\infty \ge-2C_\alpha\varepsilon^2D. 
$$
From \eqref{eq:little-o-eps-choice} and the upper bounds in \(\eqref{eq:little-o-energy}\) and \(\eqref{eq:little-o-q-sup}\), we have
$$
\left| (\Sigma(M_s,D_s)^2-1) (q_D(M_s)-\bar q_D) \right|\le \bigl(2\varepsilon S_\alpha+\varepsilon^2S_\alpha^2\bigr) C_\alpha\varepsilon^2D^\beta\le \frac{c_\alpha}{4}\varepsilon^2D^\beta. 
$$
After integrating the last estimate, we can combine with the prior two estimates to see
\begin{align}\label{estimateforD}
 \mathbb E\int_0^{\theta_D\wedge T}q_D(M_s)\,\dd s
 \ge \frac{3c_\alpha}{4}\varepsilon^2D^\beta
       \mathbb E[\theta_D\wedge T]
      -2C_\alpha\varepsilon^2D.
\end{align}

The lower bound in \eqref{estimateforD} works for the initial value $D=D_0$. Next, we allow for also later values $D_s$. The definition of \(\theta_D\) in \eqref{theta_D}  gives that 
\begin{align}\label{rD_sbound}
    \frac Dr\le D_s\le rD,
    \qquad 0\le s\le\theta_D\wedge T,
\end{align}
Hence, \eqref{eq:little-o-stability} with \(\widetilde D=D_s\) and 
with \eqref{eq:little-o-r-choice} yield
\[
    q_{D_s}(M_s)
    \ge 
    q_D(M_s)- C_\alpha\varepsilon^2D^\beta(r-1)^\alpha\ge 
    q_D(M_s)-\frac{c_\alpha}{4}\varepsilon^2D^\beta .
\]
Therefore, the bound in \eqref{estimateforD} gives
\begin{align}\label{eq:little-o-qtime}
\begin{split}
 \mathbb E\int_0^{\theta_D\wedge T}q_{D_s}(M_s)\,\dd s
 &\ge
 \mathbb E\int_0^{\theta_D\wedge T}q_D(M_s)\,\dd s
 -\frac{c_\alpha}{4}\varepsilon^2D^\beta
       \mathbb E[\theta_D\wedge T]  \\
 &\ge \frac{c_\alpha}{2}\varepsilon^2D^\beta
       \mathbb E[\theta_D\wedge T]
      -2C_\alpha\varepsilon^2D.
\end{split}
\end{align}
On the other hand, because $\Delta$ defined in \eqref{eq:Delta-def} satisfies
\[
    \mathbb E\int_0^{\theta_D\wedge T}
    \Delta(M_s,D_s)^2\,\dd s
    \le
    4\|\sigma\|_\infty^2T<\infty,
\]
we can use It\^o's isometry to see
\begin{align}\label{Ito_Isometry_bound1}
\begin{split}
\mathbb E[(D_{\theta_D\wedge T}-D)^2]
&=
\mathbb E\int_0^{\theta_D\wedge T}
\Delta(M_s,D_s)^2\,\dd s\\
&=\mathbb E\int_0^{\theta_D\wedge T}
     D_s q_{D_s}(M_s)\,\dd s\\
     &\ge \frac Dr\mathbb E\int_0^{\theta_D\wedge T}
 q_{D_s}(M_s)\,\dd s,
\end{split}
\end{align}
where the inequality uses \eqref{rD_sbound}. Using \eqref{rD_sbound} also gives $|D_{\theta_D\wedge T}-D|\le(r-1)D$, and so from
\eqref{eq:little-o-qtime} and \eqref{Ito_Isometry_bound1} we get
\[
\begin{aligned}
 (r-1)^2D^2
 &\ge \frac Dr
 \left[
   \frac{c_\alpha}{2}\varepsilon^2D^\beta
      \mathbb E[\theta_D\wedge T]
   -2C_\alpha\varepsilon^2D
 \right].
\end{aligned}
\]
Rearranging gives
\[
\mathbb E[\theta_D\wedge T]
\le
\frac{2r(r-1)^2+4C_\alpha\varepsilon^2}
     {c_\alpha\varepsilon^2}
D^{1-\beta}.
\]
Finally, we pass $T\to\infty$ and  use the monotone convergence theorem to prove \eqref{eq:little-o-exit}. In particular,
$\theta_D<\infty$ almost surely.

\noindent\textbf{Step 6/7:}  Set
\[
    \gamma:=1-\beta=2-2\alpha\in(0,1).
\]
Fix \(0<R<1/r\) and define the strictly decreasing sequence
\[
    x_k:=Rr^{-k},
    \qquad k=0,1,2,...
\]
We start the two solutions from
\[
    x:=-\frac{x_1}{2} =- \frac R{2r},
    \qquad
    y:=\frac{x_1}{2}=\frac R{2r},
\]
so that \(M_0=0\) and \(D_0=x_1 \in (x_2,x_0)\). Before giving the details, let us explain in words what our construction does. We run $D_t$ from $D_0=x_1$ at time 0 till either $x_2$ or $x_0$ is reached at $T_1$ (we will prove $T_1$ is finite). If $x_0=R$ is hit, we set $K_1:=0$ and otherwise $K_1:=2$. Because $R$ is the max value, we make it absorbing in that $K_j(\omega) = 0$ implies $K_{j+n}(\omega) =0$ for all $n\in \N$. Say that $x_2$ is hit at time $T_1$, then we run $D_t$ from $x_2$ at time $T_1$ till it exits $(x_3,x_1)$ at time $T_2$, and so forth. Here are the specifics: We set \(T_0:=0\) and \(K_0:=1\). We construct inductively \(T_j\) (nondecreasing stopping times) and \(K_j\) (nonmonotone random variables valued in $\mathbb N_0$). Suppose that \(T_j\) is an almost surely finite stopping
time, \(K_j\) is \(\mathcal F_{T_j}\)-measurable, and
\[
    D_{T_j}=x_{K_j}
    \qquad\text{on }\{K_j\ge1\}.
\]
We proceed by induction and define \(T_{j+1}:\Omega\to[0,\infty]\) by
\(T_{j+1}:=T_j\) on \(\{K_j=0\}\) and
\(T_{j+1}:=S_{j,k}\) on \(\{K_j=k\}\), \(k\in\N\), where
\[
    S_{j,k}
    :=
    \inf\left\{
        t\ge T_j:
        D_t\notin(x_{k+1},x_{k-1})
    \right\},
    \qquad
    \inf\varnothing:=\infty.
\]
Since \(D\) is continuous and adapted, \(S_{j,k}\) is a stopping time. To see that  \(T_{j+1}\) is also a stopping time,  we use 
\(\{K_j=k\}\in\mathcal F_{T_j}\), and that $T_j$ and $S_{j,k}$ are stopping times. These properties and the definition of the stopped filtration $\sF_{T_j}$  give
\[
    \{K_j=0\}\cap\{T_j\le t\}\in\mathcal F_t,\quad     \{K_j=k\}\cap\{S_{j,k}\le t\}\in\mathcal F_t,\quad t\ge0.
\]
Therefore, we get that also  \(T_{j+1}\) is a stopping time:
\[
\begin{aligned}
    \{T_{j+1}\le t\}
    ={}&
    \bigl(\{K_j=0\}\cap\{T_j\le t\}\bigr)\\
    &{}\cup
    \bigcup_{k\ge1}
    \bigl(\{K_j=k\}\cap\{S_{j,k}\le t\}\bigr)
    \in\mathcal F_t.
\end{aligned}
\]

On \(\{K_j=k\}\) for $k\in \N$, continuity of \(D\) and
\(D_{T_j}=x_{K_j} = x_k\in(x_{k+1},x_{k-1})\) give 
\[
    T_{j+1}
    =
    \inf\left\{
        t>T_j:
        D_t\in\{x_{k-1},x_{k+1}\}
    \right\}.
\]
Define
\begin{align}\label{K_j_def}
    K_{j+1}
    :=
    \begin{cases}
        0,&K_j=0,\\
        k-1,&K_j=k\ge1,\ T_{j+1}<\infty,\
              D_{T_{j+1}}=x_{k-1},\\[1mm]
        k+1,&K_j=k\ge1,\ T_{j+1}<\infty,\
              D_{T_{j+1}}=x_{k+1},\\[1mm]
        0,&T_{j+1}=\infty.
    \end{cases}
\end{align}
Then \(K_{j+1}\) is \(\mathcal F_{T_{j+1}}\)-measurable. By continuity of \(D\), we have
\[
    D_{T_{j+1}}=x_{K_{j+1}}\quad \text{on}\quad \{K_{j+1}\ge1\}.
\]
To see that \(T_{j+1}<\infty\) almost surely, we fix \(k\ge1\).
We have
\[
    \left(\frac{x_k}{r},rx_k\right)
    =(x_{k+1},x_{k-1})      \quad \text{and}\quad     D_{T_j}=x_{K_j}=x_k
 \quad \text{on}\quad \{K_{j}=k\}.
\]
We let \(\vartheta_t\) denote the shift
operator for $t\ge0$. The stopping time \(\theta_{x_k}\) in \eqref{theta_D} satisfies 
\[
\begin{aligned}
    \vartheta_{T_j}\theta_{x_k}
    &=
    \inf\left\{u\ge0:
        D_{T_j+u}\notin(x_{k+1},x_{k-1})
    \right\}=T_{j+1}-T_j \quad \text{on}\quad \{K_{j}=k\},
\end{aligned}
\]
with both sides understood as $+\infty$ when $T_{j+1}=\infty$. At this point, we do not have integrability of $T_{j+1}$  and so we cap with $n\in\N$. Because capping commutes with $\vartheta_t$, we have
\[
    \vartheta_{T_j}(\theta_{x_k}\wedge n)
    =
    (T_{j+1}-T_j)\wedge n
    \qquad\text{on }\{K_j=k\}.
\]
The strong Markov property (see, e.g.,  Eq., (7.2.5) in \O ksendal, 2003), \(\{K_j=k\}\in\mathcal F_{T_j}\), and \(D_{T_j}=x_k\) on \(\{K_j=k\}\) give 
\[
\begin{aligned}
\mathbf 1_{\{K_j=k\}}
\mathbb E\!\left[
    (T_{j+1}-T_j)\wedge n
    \,\middle|\,\mathcal F_{T_j}
\right]
&=
\mathbf 1_{\{K_j=k\}}
\mathbb E\!\left[
    \vartheta_{T_j}(\theta_{x_k}\wedge n)
    \,\middle|\,\mathcal F_{T_j}
\right]\\
&=
\mathbf 1_{\{K_j=k\}}
\mathbb E_{M_{T_j},D_{T_j}}
    [\theta_{x_k}\wedge n]\\
&=
\mathbf 1_{\{K_j=k\}}
\mathbb E_{M_{T_j},x_k}
    [\theta_{x_k}\wedge n].
\end{aligned}
\]
Since \(x_k\le R<1/r\), the bound in \eqref{eq:little-o-exit} applies. We let  \(n\to\infty\) and
use conditional monotone convergence to see
\[
    \mathbb E[T_{j+1}-T_j\mid\mathcal F_{T_j}]
    \le
    C_{\alpha,\varepsilon}x_k^\gamma
    \qquad\text{on }\{K_j=k\}.
\]
Because $k\in \N$ was arbitrary, this gives the estimate
\begin{equation}\label{eq:little-o-holding}
    \mathbb E[T_{j+1}-T_j\mid\mathcal F_{T_j}]
    \le
    C_{\alpha,\varepsilon}x_{K_j}^\gamma
    \qquad\text{on }\{K_j\ge1\}.
\end{equation}
The right-hand side is finite, so
\(T_{j+1}-T_j<\infty\) almost surely on \(\{K_j\ge1\}\).
On \(\{K_j=0\}\), we have \(T_{j+1}=T_j<\infty\) by definition. All in all, 
\(\P(T_{j+1}<\infty)=1\), which  concludes the induction step.

We define $T_\infty:=\lim_{j\to\infty}T_j$ (nondecreasing limit) and we will prove $\E[T_\infty]<\infty$. We need the transition probabilities. Fix $k\in \N$. The local martingale
\[
    \mathbf 1_{\{K_j=k\}}
    \left(
        D_{(T_j+t)\wedge T_{j+1}}-D_{T_j}
    \right)=   \mathbf 1_{\{K_j=k\}}
    \left(
        D_{(T_j+t)\wedge T_{j+1}}-x_k
    \right),
    \qquad t\ge0,
\]
is uniformly bounded, hence, is a uniformly
integrable martingale and we can apply optional sampling.  Because \(T_{j+1}<\infty\) almost surely,
conditional optional sampling yields
\[
    \mathbf 1_{\{K_j=k\}}
    \mathbb E[D_{T_{j+1}}\mid\mathcal F_{T_j}]
    =
    \mathbf 1_{\{K_j=k\}}D_{T_j}
    =
    \mathbf 1_{\{K_j=k\}}x_k.
\]
On \(\{K_j=k\}\) for $k\in\N$, we have $D_{T_{j+1}}\in \{x_{k+1},x_{k-1}\}$ and so
\[
\begin{aligned}
    x_k
    &=
    x_{k+1}
    \mathbb P(K_{j+1}=k+1\mid\mathcal F_{T_j})
    +
    x_{k-1}
    \mathbb P(K_{j+1}=k-1\mid\mathcal F_{T_j})\quad\text{on}\quad \{K_j=k\}. 
\end{aligned}
\]
Since the two probabilities sum to one and $x_{k+1}=\frac{x_k}{r}$ and $x_{k-1}=rx_k$ we have
\[
    \mathbb P(K_{j+1}=k+1\mid\mathcal F_{T_j})
    =\frac r{r+1}
    \quad\text{and}\quad
    \mathbb P(K_{j+1}=k-1\mid\mathcal F_{T_j})
    =\frac1{r+1}\quad\text{on}\quad \{K_j=k\}. 
\]

For $j \in \N_0$, we define
\[
    Y_j:=
    \begin{cases}
       x_{K_j}^\gamma,&K_j\ge1,\\
       0,&K_j=0.
    \end{cases}
\]
Because $K_0=1$, we have $Y_0 = x^\gamma_1$. For \(k\ge2\), the transition probabilities give
\[
\begin{aligned}
\mathbb E[Y_{j+1}\mid\mathcal F_{T_j}]
&=
\frac r{r+1}x_{k+1}^\gamma
+\frac1{r+1}x_{k-1}^\gamma=
\frac{r^{1-\gamma}+r^\gamma}{r+1}x_k^\gamma
\qquad\text{on }\{K_j=k\}.
\end{aligned}
\]
For \(k=1\), however, the transition \(K_{j+1}=k-1=0\) gives
\(Y_{j+1}=0\). Therefore
\[
\begin{aligned}
\mathbb E[Y_{j+1}\mid\mathcal F_{T_j}]
&=
\frac r{r+1}x_2^\gamma\\
&\le
\frac r{r+1}x_2^\gamma
+\frac1{r+1}x_0^\gamma\\
&=
\frac{r^{1-\gamma}+r^\gamma}{r+1}x_1^\gamma
\qquad\text{on }\{K_j=1\}.
\end{aligned}
\]
For $k=0$, we have \(K_{j+1}=0\) and so 
\(Y_{j+1}=Y_j=0\) on \(\{K_j=0\}\).
Because  $\gamma\in (0,1)$, the strict
concavity of  the two functions \(u\mapsto u^\gamma, u^{1-\gamma}\)  gives
\[
    r^\gamma<1+\gamma(r-1),
    \qquad
    r^{1-\gamma}<1+(1-\gamma)(r-1),\quad    m_\gamma:=\frac{r^{1-\gamma}+r^\gamma}{r+1}<1.
\]
Combining the three cases $k=0$, $k=1$, and $k\ge2$ produces
\[
    \mathbb E[Y_{j+1}\mid\mathcal F_{T_j}] \le m_\gamma \sum_{k=1}^\infty 1_{\{K_j=k\}} x_k^\gamma 
    = m_\gamma Y_j
    \qquad\text{a.s.}
\]
This implies $\E[Y_{j}] \le m_\gamma \E[Y_{j-1}] \le .... \le m^{j}_\gamma Y_0= m^{j}_\gamma x_1^\gamma$. 
Since $T_0=0$, we have $T_\infty=\sum_{j=0}^\infty (T_{j+1}-T_j)$. Furthermore, because $T_j$ is nondecreasing, Tonelli's theorem and \eqref{eq:little-o-holding} give
\[
\begin{aligned}
 \mathbb E[T_\infty]
 &=\sum_{j=0}^\infty\mathbb E[T_{j+1}-T_j]\le C_{\alpha,\varepsilon}
     \sum_{j=0}^\infty\mathbb E[Y_j]\le 
     C_{\alpha,\varepsilon}x_1^\gamma
     \sum_{j=0}^\infty m^{j}_\gamma   =  \frac{C_{\alpha,\varepsilon}x_1^\gamma}
         {1-m_\gamma}
 <\infty,
\end{aligned}
\]
because $m_\gamma \in (0,1)$. In particular, $\P(T_\infty<\infty)=1$.

\noindent\textbf{Step 7/7:} On $  \mathcal A
    :=
    \{K_j\in \N \text{ for all }j\ge0\}$, we have   \(K_{j+1}=K_j\pm1\). To rule out \(D_{T_\infty}=\ell>0\) on $  \mathcal A$, consider 
   $ D_{T_j}\ge\frac{\ell}{2}>0$ for all sufficiently large \(j\). On $\{K_j=k\}$ for some $k\in \N$, we have 
   $$
   D_{T_{j+1}} =x_{k+1} \Rightarrow |D_{T_{j+1}} -D_{T_j}| = (1-\frac1r)x_k,\quad    D_{T_{j+1}} =x_{k-1} \Rightarrow |D_{T_{j+1}} -D_{T_j}| = (r-1)x_k.
   $$
Because $(r-1) - (1-\frac1r) = \frac{(r-1)^2}{r}>0$, this gives
\[
\begin{aligned}
    |D_{T_{j+1}}-D_{T_j}|
    &\ge
    \left(1-\frac1r\right)D_{T_j}\ge
    \left(1-\frac1r\right)\frac{\ell}{2}>0,
\end{aligned}
\]
contradicting  $\lim_j D_{T_j}=\ell>0$. So, on $\sA$, path continuity gives $ D_{T_\infty}=0$.

Define
\[
    \tau_R:=\inf\{t\ge0:D_t\in\{0,R\}\}.
\]
To see $\P(\tau_R<\infty)=1$, we let $\omega\in \mathcal A^c$. Since $K_0=1$ and, before absorption, $K_{j+1}=K_j\pm1$, the first
transition to the absorbing state $0$ must be from state $1$. Then,  there is some finite integer \(j\) such that $K_{j+1}(\omega)=0$ and $K_{j}(\omega)=1$. The second line in \eqref{K_j_def} gives
 $D_{T_{j+1}}(\omega) = x_0 = R$. Alternatively, for \(\omega\in \mathcal A\), we just proved $T_{\infty}(\omega)<\infty$ and $D_{T_\infty}(\omega)=0$. Thus, $\P(\tau_R<\infty)=1$ follows.

 The stopped local martingale \(D_{t\wedge\tau_R}\) takes values in
\([0,R]\) and is therefore a uniformly integrable martingale.
Optional sampling gives
\[
\frac{R}{r}=    x_1
    =
    D_0
    =
    \mathbb E[D_{\tau_R}]
    =
    R\,\mathbb P(D_{\tau_R}=R).
\]
Up to a nullset,  we proved above
\[
    \mathcal A\subseteq\{D_{\tau_R}=0\},
    \qquad
    \mathcal A^c\subseteq\{D_{\tau_R}=R\}.
\]
We can take complements to see $
    \mathcal A=\{D_{\tau_R}=0\}$ up to a nullset 
because
$$
\   \mathcal A\supseteq\{D_{\tau_R}=R\}^c=\{D_{\tau_R}=0\}.
$$
Because probabilities must add to one, this implies
\begin{equation}\label{eq:little-o-collision-probability}
    \mathbb P(\mathcal A)
    =
    \mathbb P(D_{\tau_R}=0)
    =
    1-\frac{x_1}{R}
    =
    1-\frac1r
    >0.
\end{equation}
On $\mathcal A=\{D_{\tau_R}=0\}$, we have
\[
    X_{T_\infty}^{-\frac{x_1}2}
    =
    X_{T_\infty}^{\frac{x_1}2},
\]
and pathwise uniqueness implies that the two solutions remain together
thereafter. Consequently,
\[
    \mathbb P\!\left(
        X_t^{-\frac{x_1}2}=X_t^{\frac{x_1}2}
        \text{ for some }t<\infty
    \right)
    \ge
    1-\frac1r
    >0.
\]

\end{proof}

\appendix

\section{Examples related to Yamada (1986)}

In our setting with positive continuous coefficients satisfying
Assumption~\ref{ass1}, with strict comparison considered up to
the minimum of the lifetimes, this appendix illustrates
\[
    W^{1,2}_{\mathrm{loc}}
    \subsetneq
    \{\text{Yamada (1986)}\}  \subsetneq  \{\text{strict comparison holds}\}.
\]
Appendix \ref{outsideW1p} gives the first strict inclusion and  Appendix \ref{Yamada_fails} gives the second strict inclusion.

\subsection{A continuous coefficient outside $W^{1,2}_{\mathrm{loc}}$ for which strict comparison holds}\label{outsideW1p}

This  example shows the
coefficients covered by \eqref{Yamada1986} strictly contain \(W^{1,2}_{\mathrm{loc}}\). We consider the $2\pi$-periodic coefficient
\begin{equation}\label{eq:example-lacunary}
    \sigma_{\mathrm{lac}}(x)
    :=
    2+\sum_{n=1}^{\infty}
      \frac{2^{-n}}{\sqrt n}\cos(2^n x),\quad x\in \R.
\end{equation}
The series converges uniformly and
$1<\sigma_{\mathrm{lac}}<3$.  Item~7 in V.5.7 in Zygmund (1959) ensures that, for a lacunary
trigonometric series
\[
\sum_{k=1}^\infty\big(a_k\cos(n_kx)+b_k\sin(n_kx)\big),
\qquad
\inf_{k\in \N}\frac{n_{k+1}}{n_k}>1,
\]
differentiability on a set of positive measure implies
$
\sum_k n_k^2(a_k^2+b_k^2)<\infty.
$
For \eqref{eq:example-lacunary}, we have \(n_k:=2^k\), \(a_k:=2^{-k}/\sqrt{k}\), and \(b_k:=0\), which give
\[
    \sum_{k=1}^\infty n_k^2a_k^2
    =
    \sum_{k=1}^\infty\frac1k
    =
    \infty.
\]
Therefore, \(\sigma_{\mathrm{lac}}\) is nondifferentiable almost everywhere and so $\sigma_{\mathrm{lac}}\notin W^{1,1}_{\mathrm{loc}}(\mathbb R)$. We set $h:=|x-y|\le\frac12$ and choose $N$ so
that $2^{-(N+1)}<h\le2^{-N}$.  Then, we have
\[
\begin{aligned}
|\sigma_{\mathrm{lac}}(x)-\sigma_{\mathrm{lac}}(y)|
&\le
h\sum_{n\le N}\frac1{\sqrt n}
+2\sum_{n>N}\frac{2^{-n}}{\sqrt n}\le
C h\sqrt{\log\frac eh}.
\end{aligned}
\]
Therefore, on a neighborhood of zero, we can use $    \rho(h):=C h\sqrt{\log\frac eh}$ as modulus. 
Consequently, $\sigma_{\mathrm{lac}} \in C^\frac12(\R)$ and even $\sigma_{\mathrm{lac}} \in C^\beta(\R)$ for all $\beta \in (0,1)$. Inserting $\rho$ into \eqref{YOcond1} gives 
\[
\begin{aligned}
    \int_{0+}\frac{h\,\dd h}{\rho(h)^2}
    &=
    \frac1{C^2}
    \int_{0+}
    \frac{\dd h}{h\log(e/h)}
    =
    \infty,
\end{aligned}
\]
and strict comparison follows.

\subsection{A continuous coefficient outside Yamada (1986) for which strict comparison holds}\label{Yamada_fails}
The coefficients in \eqref{ex_main} are not covered by \eqref{Yamada1986}. To see this, we consider $\sigma_{\text{sym}}$ in  \eqref{ex_main} and assume to the contrary that, for some small $x_0>0$ we have
\[
    1+\sqrt{|x|}
    =
    \sigma_1(t,x)\sigma_2(x),\quad |x|< x_0,
\]
satisfies \eqref{Yamada1986}.  Fix \(t\). The modulus hypothesis makes
\(x\mapsto\sigma_1(t,x)\) continuous, while \(\sigma_2\) is assumed continuous
(even absolutely continuous). Since
\[
    \sigma_1(t,0)\sigma_2(0)=\sigma(0)=1,
\]
both factors are nonzero at zero. Hence, after reducing $x_0$ if necessary, there are constants \(0<C<M<\infty\) such
that
\[
    C\le
    |\sigma_1(t,x)|,\ |\sigma_2(x)|
    \le M.
\]
Since $\sigma_2'\in L^2_{\mathrm{loc}}$, Cauchy--Schwarz's inequality gives
\[
\begin{aligned}
    |\sigma_2(h)-\sigma_2(0)|
    &\le
    \sqrt h
    \left(\int_0^h|\sigma_2'(u)|^2\,\dd u\right)^{\frac12}=o(\sqrt h),
    \qquad h\downarrow0.
\end{aligned}
\]
On the other hand,
\[
\begin{aligned}
    \sqrt h
    &=
    \sigma(h)-\sigma(0)=
    \sigma_2(h)\bigl(\sigma_1(t,h)-\sigma_1(t,0)\bigr)
    +
    \sigma_1(t,0)\bigl(\sigma_2(h)-\sigma_2(0)\bigr).
\end{aligned}
\]
Hence, for all sufficiently small $h>0$, we have
\[
    |\sigma_1(t,h)-\sigma_1(t,0)|
    \ge c\sqrt h
\]
for some $c>0$.  Therefore, any modulus $\rho$ satisfying
\[
    |\sigma_1(t,x)-\sigma_1(t,y)|
    \le \rho(|x-y|)
\]
must satisfy $\rho(h)\ge c\sqrt h$ for small $h$.  Consequently, 
\[
    \int_{0+}\frac{h\,dh}{\rho(h)^2}
    \le
    \frac1{c^2}\int_{0+}dh
    <\infty,
\]
and \eqref{YOcond1} cannot hold. The same argument applies to  $\sigma_+$.

\section{The exponential distribution of \(L_{\tau_W}^0(W)\)}\label{Bass}

Fix \(a>0\) and define
\[
    u_a(x):=\frac{1+a|x|}{1+a},
    \qquad |x|\le1.
\]
By the It\^o--Tanaka formula,
\[
    \dd u_a(W_t)
    =
    \frac{a}{1+a}\operatorname{sgn}(W_t)\,\dd W_t
    +
    \frac{a}{1+a}\,\dd L_t^0(W),\quad t<\tau_W.
\]
Consequently, integration by parts gives
\[
\begin{aligned}
    \dd\!\left(
        e^{-aL_t^0(W)}u_a(W_t)
    \right)
    &=
    e^{-aL_t^0(W)}
    \frac{a}{1+a}\operatorname{sgn}(W_t)\,\dd W_t
    \\
    &\quad+
    e^{-aL_t^0(W)}
    \left(
        \frac{a}{1+a}
        -
        a u_a(W_t)
    \right)\dd L_t^0(W),\quad t<\tau_W.
\end{aligned}
\]
Since \(\dd L_t^0(W)\) is supported on the set
\(\{t\ge0:W_t=0\}\), and
\[
    u_a(0)=\frac1{1+a},\quad a>0,
\]
the finite-variation term vanishes. Hence
\[
    e^{-aL_{t\wedge\tau_W}^0(W)}
    u_a(W_{t\wedge\tau_W})
\]
is a bounded martingale valued between 0 and 1. Optional Stopping and the Dominated Convergence Theorems produce
\[
    u_a(0)
    =
    \E\!\left[
        e^{-aL_{\tau_W}^0(W)}u_a(W_{\tau_W})
    \right].
\]
Because \(|W_{\tau_W}|=1\) and \(u_a(\pm1)=1\), we obtain the exponential Laplace transform:
\[
    \E\!\left[e^{-aL_{\tau_W}^0(W)}\right]
    =
    \frac1{1+a},
    \qquad a>0.
\]

\newpage
\noindent Declaration of interest: Kasper Larsen has no conflicts of interest.  \ \\

\noindent Declaration of generative AI and AI-assisted technologies in the manuscript preparation process: The author used ChatGPT to obtain feedback on exposition and to check mathematical arguments for correctness. Any and all content has been manually verified with pen and paper by the author. The author  assumes full responsibility for all content.\ \\

\noindent Data availability statement: We do not analyze nor generate any datasets.

\end{document}